\documentclass[12pt]{amsart}
\usepackage[margin=1in]{geometry}
\usepackage{amsmath}
\usepackage{amssymb}
\usepackage{amsfonts}
\usepackage{amsthm}
\usepackage{mathrsfs}
\usepackage{enumerate}
\usepackage[hidelinks]{hyperref}

\hypersetup{
    pdftitle={Finite Good Witnesses for Generalized Curve Projections at the Rectifiable Endpoint},
    pdfauthor={Caleb Marshall}
}

\newtheorem{theorem}{Theorem}[section]
\newtheorem{proposition}[theorem]{Proposition}
\newtheorem{lemma}[theorem]{Lemma}
\newtheorem{corollary}[theorem]{Corollary}

\theoremstyle{definition}
\newtheorem{definition}[theorem]{Definition}
\newtheorem{example}[theorem]{Example}
\newtheorem{remark}[theorem]{Remark}

\numberwithin{equation}{section}

\DeclareMathOperator{\vspan}{Span}
\DeclareMathOperator{\rank}{rank}
\DeclareMathOperator{\Aff}{Aff}

\begin{document}
\title[Finite Good Witnesses for Generalized Curve Projections]{Finite Good Witnesses for Generalized Curve Projections at the Rectifiable Endpoint}

\author[C.~Marshall]{Caleb Marshall}
\address{Department of Mathematics \\ The University of Toronto, St. George Campus, Toronto, ON, Canada}
\email{caleb.marshall@utoronto.ca}

\date{August 10, 2026}

\begin{abstract}
For a $1$-rectifiable set $E \subset \mathbb{R}^d$ of positive length, a theorem of Federer shows that, among any $d$ linearly independent orthogonal projections of $E$, at least one has positive length. We develop a version of this finite-witness principle for generalized curve projections.

Given scalar-valued mappings
$
\varphi_1,\ldots,\varphi_d:\mathbb{R}^d\longrightarrow\mathbb{R},
$
we introduce the canonical encoding map
$
\mathsf{H}:=(\varphi_1,\ldots,\varphi_d).
$
Where $D\mathsf{H}$ is invertible, a local bilipschitz change of variables and Federer's theorem show that $\varphi_j(E)$ has positive length for some $j$. If a positive-length portion of $E$ lies in the critical set, we instead argue intrinsically on a $C^1$ hypersurface containing it, provided that $\mathsf{H}$ retains full tangential rank there. This yields an abundance of deterministic finite witnesses, as well as structural bounds for those exceptional parameters where the good witness property fails. A new fold non-degeneracy condition makes our constructions stable under perturbations of the underlying parameters. At this rectifiable endpoint, these conclusions complement work of Peres--Schlag, which developed exceptional set estimates for fractal sets of Hausdorff dimension strictly greater than one.

We then apply our framework to several nonlinear projection families. Affinely independent pinned squared-distance maps satisfy the tangential and fold conditions, whereas two planar radial projections lose tangential rank along their critical line. We also verify the hypotheses for more exotic examples, including nonlinear anisotropic distances, Bregman functionals arising from smooth approximations of polyhedral norms, and families whose critical hypersurfaces have prescribed $C^2$ geometry.
\end{abstract}

\subjclass[2020]{Primary 28A75; Secondary 28A78}
\keywords{rectifiable sets, generalized projections, projection theorems, exceptional set estimates, fold singularities}

\maketitle

\section{Introduction}

Let $d \geq 2$ and let $\mathcal{H}^1$ denote the $1$-dimensional Hausdorff measure on $\mathbb{R}^d$. A Borel set $E\subseteq\mathbb{R}^d$ is said to be \emph{$1$-rectifiable} if
$$
\mathcal{H}^1\bigg(E\setminus\bigcup_{j=1}^{\infty}f_j(\mathbb{R})\bigg)=0
\qquad\text{for Lipschitz maps }f_j : \mathbb{R} \rightarrow \mathbb{R}^d.
$$
Conversely, a Borel set $E \subset \mathbb{R}^d$ is \textit{purely $1$-unrectifiable} if $\mathcal{H}^1 (E \cap f(\mathbb{R})) = 0$ for every Lipschitz function $f: \mathbb{R} \rightarrow \mathbb{R}^d$. In this article, we study the relationship between the length of $1$-rectifiable sets of positive, finite $\mathcal{H}^1$ measure (which we refer to as $\boldsymbol{1}$-\textbf{sets}) and the $\mathcal{H}^1$ measure of their projections onto curves, by developing a general framework for determining whether a given family of $d$ scalar mappings has the following property.

\begin{definition}[Good Witness Property]
\label{def:goodwitnessproperty-intro}
Let
$$
\mathsf{\Phi}:=\{\varphi_1,\ldots,\varphi_d\}
$$
be a family of mappings $\varphi_j:\mathbb{R}^d\rightarrow\mathbb{R}$. We say that $\mathsf{\Phi}$ has the \textbf{good witness property for rectifiable $1$-sets} (or, simply, the \textbf{good witness property}) if, for every rectifiable $1$-set $E\subseteq\mathbb{R}^d$, there exists an index $j_E\in\{1,\ldots,d\}$ such that
$$
\mathcal{H}^1\big(\varphi_{j_E}(E)\big)>0.
$$
The mapping $\varphi_{j_E}$ is called a \textbf{good witness for $E$}.
\end{definition}

The canonical classical example of this phenomenon is Federer's Projection Theorem, which relates the $\mathcal{H}^1$ measure of a given $1$-rectifiable set to the $\mathcal{H}^1$ measure of its projections onto any given orthonormal frame. We phrase Federer's result in language that is best adapted for our further use, and refer the reader to Federer's seminal textbook (see \cite[Theorem 3.2.27]{Federer1996}) for a full exposition of this result.

\begin{theorem}[Federer's Projection Theorem]\label{thm:federerprojection-intro}
Let $E \subset \mathbb{R}^d$ be any $1$-rectifiable set, and let $e_1,\ldots,e_d$ denote any orthonormal basis of $\mathbb{R}^d$. Writing $x=\sum_i x_i e_i$, let
$$
\pi_j(x):=x_j,
\qquad j\in\{1,\ldots,d\}.
$$
If
$$
a_j:=\int_{\pi_j(E)}
\mathcal{H}^0\big(\pi_j^{-1}(y)\cap E\big)\,d\mathcal{H}^1(y),
$$
then
$$
\big(a_1^2+\cdots+a_d^2\big)^{1/2}
\leq \mathcal{H}^1(E)
\leq a_1+\cdots+a_d.
$$
\end{theorem}

If $E$ is a rectifiable $1$-set, the right-hand inequality forces $a_j>0$ for at least one $j$, and hence $\mathcal{H}^1\big(\pi_j(E)\big)>0$. Thus, the coordinate projections associated with every orthonormal basis have the good witness property. Our purpose is to develop a nonlinear version of this finite-witness phenomenon for generalized curve projections.

To this end, for each family $\mathsf{\Phi}:=\{\varphi_1,\ldots,\varphi_d\}$ of $C^2$ scalar mappings $\varphi_j:\mathbb{R}^d\rightarrow\mathbb{R}$, define its associated \textbf{canonical encoding map} as
$$
\mathsf{H}_{\mathsf{\Phi}}(x)
:=
\big(\varphi_1(x),\ldots,\varphi_d(x)\big).
$$
We denote its associated set of critical points by
$$
\Sigma_{\mathsf{\Phi}}
:=
\big\{
x\in\mathbb{R}^d:
\rank D_x\mathsf{H}_{\mathsf{\Phi}}(x)\leq d-1
\big\}.
$$

A common principle underlying this work is as follows: away from the critical points of $\mathsf{H}_{\mathsf{\Phi}}$, the Inverse Function Theorem makes the encoding locally bilipschitz, after which Federer's Projection Theorem may be applied to its coordinate functions. Thus the good witness property follows whenever a positive-length portion of $E$ lies in the regular region. We will prove this result for an arbitrary collection of $d$ scalar maps $\varphi_1,\ldots,\varphi_d$ in Section \ref{subsec:regularcurveprojections-genexample}.

The principal difficulty is that a positive-length portion of a rectifiable set may instead lie inside the critical set $\Sigma_{\mathsf{\Phi}}$. In order to address this further complication, we introduce the following two differential-geometric conditions.

\begin{definition}[Critical-Hypersurface and Tangential-Immersion Conditions]
\label{def:criticalhypersurfacetangentialimmersion-genexample}
Let 
$$
\mathsf{\Phi}:=\{\varphi_1,\ldots,\varphi_d\}
$$
be a collection of $C^2$ mappings $\varphi_j:\mathbb{R}^d\rightarrow\mathbb{R}$ with canonical encoding map $\mathsf{H}_{\mathsf{\Phi}}:\mathbb{R}^d\rightarrow\mathbb{R}^d$.

\begin{enumerate}
    \item The family $\mathsf{\Phi}$ satisfies the \textbf{critical-hypersurface condition} if there exists a $C^1$ embedded submanifold $M\subset\mathbb{R}^d$ satisfying $\dim M=d-1$ and
    $$
    \Sigma_{\mathsf{\Phi}}
    =
    \big\{
    x\in\mathbb{R}^d:
    \rank D_x\mathsf{H}_{\mathsf{\Phi}}(x)\leq d-1
    \big\}
    \subseteq M.
    $$

    \item Relative to such a hypersurface $M$, the family $\mathsf{\Phi}$ satisfies the \textbf{tangential-immersion condition} if, for every $x\in M$,
    $$
    \rank\big(
    D_x\mathsf{H}_{\mathsf{\Phi}}(x)\vert_{T_xM}
    \big)
    =d-1.
    $$
    Equivalently, for every $x\in M$, the linear functionals
    $$
    D_x\varphi_j(x)\vert_{T_xM}:T_xM\rightarrow\mathbb{R},
    \qquad j=1,\ldots,d,
    $$
    span the cotangent space $(T_xM)^*$.
\end{enumerate}
The family $\mathsf{\Phi}$ is called \textbf{admissible} if there exists a $C^1$ hypersurface $M$ relative to which both conditions hold.
\end{definition}

\begin{theorem}[Finite Witnesses Through the Critical Set]
\label{thm:dcurvesrefined-genexample}
Let $\mathsf{\Phi}:=\{\varphi_1,\ldots,\varphi_d\}$ be an admissible family of mappings $\varphi_j:\mathbb{R}^d\rightarrow\mathbb{R}$, where $d\geq2$. 

Then every Borel $1$-rectifiable set $E\subseteq\mathbb{R}^d$ satisfying $\mathcal{H}^1(E)>0$ has a good witness; that is, there exists some $j_E\in\{1,\ldots,d\}$ such that
$$
\mathcal{H}^1\big(\varphi_{j_E}(E)\big)>0.
$$
In particular, $\mathsf{\Phi}$ has the good witness property for rectifiable $1$-sets.
\end{theorem}

We remark that the standard convention that the empty set is an embedded hypersurface of $\mathbb{R}^d$ is adopted. Thus, when $\Sigma_{\mathsf{\Phi}}=\emptyset$, both conditions hold vacuously with $M=\emptyset$. The tangential-immersion terminology is quite literal, since the rank condition states precisely that $\mathsf{H}_{\mathsf{\Phi}}\vert_M$ is a $C^1$ immersion.

By definition, the critical-hypersurface and tangential-immersion conditions are pointwise geometric hypotheses. However, as we shall show by direct example, these conditions need not be open under perturbation of the underlying parameters.

For this work, the additional condition which provides this stability under open perturbations is the following fold non-degeneracy, which we state relative to an arbitrary underlying parameter space $\mathcal{A} \subseteq \mathbb{R}^m$ for some $m \geq 1$.

\begin{definition}[Fold Non-Degeneracy]
\label{def:foldnondegeneracy-genexample}
Let $\mathcal{A}\subseteq\mathbb{R}^m$ be an open parameter set. Suppose
$$
(\alpha,x)\longmapsto\varphi_\alpha(x)
$$
is a $C^2$ mapping on $\mathcal{A}\times\mathbb{R}^d$. For each parameter tuple $\boldsymbol{\alpha}=(\alpha_1,\ldots,\alpha_d)\in\mathcal{A}^d$, define the associated canonical encoding map
$$
\mathsf{H}_{\boldsymbol{\alpha}}(x)
:=
\big(
\varphi_{\alpha_1}(x),\ldots,\varphi_{\alpha_d}(x)
\big),
$$
and further let
$$
\Sigma_{\boldsymbol{\alpha}}
:=
\big\{
x\in\mathbb{R}^d:
\rank D_x\mathsf{H}_{\boldsymbol{\alpha}}(x)\leq d-1
\big\},
\qquad\text{and let}
J_{\boldsymbol{\alpha}}(x)
:=
\det D_x\mathsf{H}_{\boldsymbol{\alpha}}(x).
$$
If $U\subseteq\mathbb{R}^d$ is any open set, the encoding $\mathsf{H}_{\boldsymbol{\alpha}}$ is said to be \textbf{fold non-degenerate on $U$} if every $x\in U\cap\Sigma_{\boldsymbol{\alpha}}$ satisfies the following conditions simultaneously:
\begin{enumerate}
    \item The \textbf{rank-stability condition}, which means that
\begin{equation}\label{eq:foldrankstability-genexample}
\rank D_x\mathsf{H}_{\boldsymbol{\alpha}}(x)=d-1.
\end{equation}
\item The \textbf{kernel-transversality condition}, which means that
\begin{equation}\label{eq:foldkerneltransversality-genexample}
D_xJ_{\boldsymbol{\alpha}}(x)[v]\neq0,
\qquad \forall v\in
\ker D_x\mathsf{H}_{\boldsymbol{\alpha}}(x) \setminus \{0\}.
\end{equation}
\end{enumerate}
\end{definition}

The geometric interpretation of fold non-degeneracy is as follows. First, the condition \eqref{eq:foldrankstability-genexample} means that, for each critical point $x \in U\cap\Sigma_{\boldsymbol{\alpha}}$, the space $\ker D_x \mathsf{H}_{\boldsymbol{\alpha}}(x)$ is one-dimensional and hence determines a unique direction, represented up to sign by a unit vector $n_x$. Condition \eqref{eq:foldkerneltransversality-genexample} then requires that the Jacobian determinant $J_{\boldsymbol{\alpha}}$ vanish transversely along the unique direction lost by the differential of a corank-one encoding. We remark that Condition \eqref{eq:foldkerneltransversality-genexample} is a standard fold condition for a mapping between spaces of the same dimension; see, for example, \cite{GG1973}.

Our main motivation for introducing fold non-degeneracy is the following openness result for the good witness property, under the assumption that the underlying family has the fold non-degeneracy property in a compact neighbourhood of the rectifiable set $E$.

\begin{theorem}[Stable Finite Witnesses]
\label{thm:introstablefinitewitnesses}
Let $\mathcal{A}\subseteq\mathbb{R}^m$ be open, suppose $(\alpha,x)\mapsto\varphi_\alpha(x)$ is jointly $C^2$, and fix a tuple $\boldsymbol{\alpha}^0=(\alpha_1^0,\ldots,\alpha_d^0)\in\mathcal{A}^d$. Let $K\subset U\subseteq\mathbb{R}^d$, where $K$ is compact and $U$ is open, and suppose that $\mathsf{H}_{\boldsymbol{\alpha}^0}$ is fold non-degenerate at every critical point in $\overline{U}$. Then there are open neighbourhoods
$$
\alpha_j^0\in U_j\subseteq\mathcal{A},
\qquad j=1,\ldots,d,
$$
with the following property: for every Borel $1$-rectifiable set $E\subseteq K$ satisfying $\mathcal{H}^1(E)>0$, there is an index $j_E\in\{1,\ldots,d\}$ such that
$$
\mathcal{H}^1\big(\varphi_\alpha(E)\big)>0
\qquad\text{for every }\alpha\in U_{j_E}.
$$
The neighbourhoods depend only on $\boldsymbol{\alpha}^0$ and $K$; only the successful index may depend on $E$.
\end{theorem}

The proof will show that, on a fixed compact set $K\subset\mathbb{R}^d$, fold non-degeneracy persists under sufficiently small changes of the parameter tuple. We will further demonstrate that, under a corresponding \textit{global} fold-stability hypothesis, the same conclusion holds without compactly localizing $E$.

The critical-hypersurface, tangential-immersion and fold non-degeneracy properties are all motivated by two geometrically simple families of curve projections: pinned squared-distance maps and planar radial angle maps. In particular, these initial examples will be used to clarify the role of the three conditions above, and to give nice working examples of each of these properties in action. We consider both examples independently, and before any proofs of our more general results, in Section \ref{sec:pinneddistances-examples} and Section \ref{sec:radialprojections-examples}, respectively.

\subsection{Background and related work}
\label{subsec:backgroundrelatedwork-intro}

The classical projection theory of planar $1$-sets began with Besicovitch's Projection Theorem. In its planar form, the theorem states that if a Borel set $E\subseteq\mathbb{R}^2$ satisfies
$$
0<\mathcal{H}^1(E)<\infty
$$
and is purely $1$-unrectifiable, then
$$
\mathcal{H}^1\big(P_L(E)\big)=0
$$
for almost every line $L\in G(2,1)$, where $P_L$ denotes orthogonal projection onto $L$. Besicovitch proved this result in \cite{Besicovitch1939}, and Federer subsequently extended it to $m$-dimensional sets in $\mathbb{R}^d$. The resulting Besicovitch--Federer Projection Theorem characterizes purely $m$-unrectifiable sets of finite $\mathcal{H}^m$ measure by the almost-everywhere vanishing of their orthogonal projections; see \cite{Federer1996} and \cite[Theorem 18.1]{Mattila1995}.

For the present paper, a finite consequence of Federer's theorem is especially relevant. Let $L_1,\ldots,L_d\subseteq\mathbb{R}^d$ be linearly independent lines, and let $P_{L_j}$ denote orthogonal projection onto $L_j$. Then every Borel $1$-rectifiable set $E\subseteq\mathbb{R}^d$ satisfying $\mathcal{H}^1(E)>0$ obeys
$$
\mathcal{H}^1\big(P_{L_j}(E)\big)>0
$$
for at least one $j\in\{1,\ldots,d\}$. Indeed, if $e_j$ is a unit vector spanning $L_j$, then
$$
T(x):=\big(\langle x,e_1\rangle,\ldots,\langle x,e_d\rangle\big)
$$
is an invertible linear mapping, and Federer's Projection Theorem applied to the rectifiable set $T(E)$ gives the conclusion; see Theorem~\ref{thm:federerprojection-intro} and \cite[Theorem 3.2.27]{Federer1996}. We refer to this observation as the \emph{$d$-lines corollary of Federer's theorem}. It is the deterministic finite-witness counterpart of the almost-everywhere projection theorem and the classical linear model which motivated the good witness property developed in this work.

For nonlinear parameterized families, a complementary line of research studies almost-everywhere projection theorems and quantitative estimates for exceptional parameters. Peres and Schlag developed a general transversality framework for such families by studying the normalized two-point difference quotient
$$
(\alpha;x,y)\longmapsto
\frac{\Pi_\alpha(x)-\Pi_\alpha(y)}{|x-y|},
$$
where $x,y\in\mathbb{R}^d$ and $\alpha\in\mathcal{A}\subseteq\mathbb{R}^m$. Their principal condition requires the parameter dependence to separate points whose images are suitably close. Together with appropriate smoothness assumptions, this converts energy estimates for a measure into almost-everywhere projection theorems and Hausdorff-dimension estimates for exceptional parameters \cite{PS2000}. Orponen subsequently obtained sharp exceptional-parameter estimates for slicing problems in this generalized-projection setting \cite{Orponen2014}.

Among the developments closest to the examples in this paper, Bongers and Taylor used quantitative transversality to obtain lower bounds for nonlinear variants of Favard length, including curve projections, radial visibility, and higher-dimensional surface analogues \cite{BT2023}. A recent comparison principle of \L aba, McDonald, and Taylor locally relates generalized projections to orthogonal projections and transfers quantitative Favard-length estimates between the two settings \cite{LMT2026}. These results place pinned-distance, radial-projection, and variable-radii circle problems within a broader nonlinear projection scheme. The emphasis in the present paper is different: rather than averaging over parameter spaces, it studies a finite-witness problem at the $1$-rectifiable dimensional endpoint. The objective is to find, from a prescribed collection of $d$ scalar mappings, a concrete index for which the image of a given positive-length rectifiable set has positive length.

More recently, Li and Taylor established a quantified two-projection theorem for nonlinear projection families, with applications that include pinned distances and radial projections \cite{LT2026}. Their quantitative parameter-space approach is complementary to the deterministic finite-tuple method developed here. Fraser's nonlinear projection theorem for Assouad dimension provides another complementary framework, with applications to distance sets, radial projections, and distances induced by curved norms \cite{Fraser2023}.

This paper was developed in tandem with the companion paper \cite{BBMTCompanion}. The two papers share the same finite nonlinear encoding principle but pursue complementary aims: the companion paper applies it to pinned distance sets, radial projections, and unions of circles with variable radii, while the present paper develops the \textit{critical-hypersurface}, \textit{tangential-immersion}, and \textit{fold non-degeneracy conditions}. These differential-geometric conditions do not appear in the companion paper, and we posit that these more general differential-geometric conditions provide a unified framework for the same three canonical examples presented in \cite{BBMTCompanion}.

\subsection*{Organization of the paper}

These two motivating families are presented before the main results from the abstract are proven. Section~\ref{sec:pinneddistances-examples} verifies the critical-set, tangential-rank, fold, and global openness properties for pinned distances, while Section~\ref{sec:radialprojections-examples} identifies the corresponding failure and localized openness phenomena for radial projections. Their contrast motivates the hypotheses used throughout the remainder of the paper. Section~\ref{sec:generalizedcurveprojections-genexample} then develops the finite-witness framework, the corresponding exceptional-parameter estimates, and the local and global parameter-space openness supplied by fold non-degeneracy. Section~\ref{sec:exoticcurveprojections-esoteric} treats more exotic examples of generalized curve projections, including: pinned nonlinear anisotropic distance functions, Bregman functionals associated with smoothed polyhedral norms, and critical hypersurfaces with arbitrary prescribed $C^2$ geometry. We also include a short appendix, which records the differential-geometric facts used in our proofs in language that is easily adapted for our arguments.

\subsection*{Acknowledgments}
This work grew out of experiences fostered by the American Institute of Mathematics through the SQuaREs project \emph{Covering Fractals by Curves}, with group members P. Bright, R. Bongers and K. Taylor. The author is supported by an \emph{Arts and Science Postdoctoral Fellowship} at the University of Toronto as well as a \textit{Canadian Postdoctoral Research Award} from the Natural Sciences and Engineering Research Council of Canada. The author is also grateful to Alex Iosevich and Ignacio Uriarte-Tuero for ongoing, engaging mathematical discussions.

\section{Pinned distances}
\label{sec:pinneddistances-examples}

This section determines exactly how the pinned-distance family interacts with the critical-hypersurface, tangential-immersion, and fold non-degeneracy conditions introduced in Definitions~\ref{def:criticalhypersurfacetangentialimmersion-genexample} and \ref{def:foldnondegeneracy-genexample}. These conditions are specific to the present paper. The calculations below verify them for pinned squared-distance maps; the general proofs appear in Section~\ref{sec:generalizedcurveprojections-genexample}.

Let
$$
\mathcal{P}
:=
\{p_1,\ldots,p_d\}\subseteq\mathbb{R}^d
$$
be an affinely independent set. Thus,
$$
L_{\mathcal{P}}
:=
\vspan\{p_2-p_1,\ldots,p_d-p_1\}
$$
has dimension $d-1$, and
$$
M_{\mathcal{P}}
:=
\Aff(\mathcal{P})
=
p_1+L_{\mathcal{P}}
$$
is an affine hyperplane. For each pin $p\in\mathbb{R}^d$, write
$$
d_p(x):=|x-p|
\qquad\text{and}\qquad
\Delta_p(E):=d_p(E)=\big\{|x-p|:x\in E\big\}.
$$
For each $p_j\in\mathcal{P}$, set
$$
\varphi_{p_j}(x):=d_{p_j}(x)^2=|x-p_j|^2,
$$
and form the canonical encoding map
$$
\mathsf{H}_{\mathcal{P}}(x)
:=
\big(\varphi_{p_1}(x),\ldots,\varphi_{p_d}(x)\big).
$$

\begin{proposition}[Admissibility of Pinned Squared Distances]
\label{prop:pinneddistanceadmissible-examples}
Let $\mathcal{P}\subseteq\mathbb{R}^d$ be any affinely independent set of $d$ pins. Then
$$
\Sigma_{\mathcal{P}}
:=
\big\{x\in\mathbb{R}^d:
\rank D_x\mathsf{H}_{\mathcal{P}}(x)\leq d-1\big\}
=
M_{\mathcal{P}}.
$$
Moreover,
$$
\rank\big(
D_x\mathsf{H}_{\mathcal{P}}(x)
\vert_{T_xM_{\mathcal{P}}}
\big)
=d-1,
\qquad x\in M_{\mathcal{P}}.
$$
Consequently, the pinned squared-distance family satisfies the critical-hypersurface and tangential-immersion conditions. Moreover, $\mathsf{H}_{\mathcal{P}}$ is fold non-degenerate on $\mathbb{R}^d$.
\end{proposition}

\begin{proof}
The gradients of the coordinate functions form the rows of the derivative matrix, and we thus have that
$$
D_x\mathsf{H}_{\mathcal{P}}(x)
=
2\begin{pmatrix}
(x-p_1)^{\mathsf T}\\
\vdots\\
(x-p_d)^{\mathsf T}
\end{pmatrix}.
$$
Subtracting the first row from each of the remaining rows gives
\begin{equation}\label{eq:pinneddistancerowreduction-examples}
\det D_x\mathsf{H}_{\mathcal{P}}(x)
=
2^d
\det\begin{pmatrix}
(x-p_1)^{\mathsf T}\\
(p_1-p_2)^{\mathsf T}\\
\vdots\\
(p_1-p_d)^{\mathsf T}
\end{pmatrix}.
\end{equation}
Let $n_{\mathcal{P}}$ be a unit normal to $L_{\mathcal{P}} := \vspan \{p_2 - p_1,\ldots,p_d - p_1\}$. Since $p_1-p_2,\ldots,p_1-p_d$ form a basis of $L_{\mathcal{P}}$, expansion of \eqref{eq:pinneddistancerowreduction-examples} in the normal direction shows that there is a nonzero constant $c_{\mathcal{P}}$ such that
\begin{equation}\label{eq:pinneddistancejacobian-examples}
J_{\mathcal{P}}(x)
:=
\det D_x\mathsf{H}_{\mathcal{P}}(x)
=
c_{\mathcal{P}}
\langle x-p_1,n_{\mathcal{P}}\rangle.
\end{equation}
It follows immediately that $\Sigma_{\mathcal{P}}=M_{\mathcal{P}}$.

Fix $x\in M_{\mathcal{P}}$. Since the rows of $D_x\mathsf{H}_{\mathcal{P}}(x)$ all lie in $L_{\mathcal{P}}$, whereas their pairwise differences span $L_{\mathcal{P}}$, the matrix $D_x\mathsf{H}_{\mathcal{P}}(x)$ has rank $d-1$. More precisely, if $v\in T_xM_{\mathcal{P}}=L_{\mathcal{P}}$ and
$$
D_x\mathsf{H}_{\mathcal{P}}(x)[v]=0,
$$
then subtraction of the first coordinate from the others gives
$$
\langle p_j-p_1,v\rangle=0,
\qquad \forall j \in \{2,\ldots,d\}.
$$
Since these differences span $L_{\mathcal{P}}$, this then implies that $v =0$. Thus the restriction to $T_xM_{\mathcal{P}}$ has rank $d-1$, which is precisely the tangential-immersion condition.

Finally, we verify the fold non-degeneracy condition. Since
$$
\ker D_x\mathsf{H}_{\mathcal{P}}(x)
=
\vspan\{n_{\mathcal{P}}\},
\qquad \forall x\in M_{\mathcal{P}},
$$
the identity \eqref{eq:pinneddistancejacobian-examples} gives
$$
D_xJ_{\mathcal{P}}(x)[n_{\mathcal{P}}]
=c_{\mathcal{P}}\neq0,
$$
and the same non-vanishing holds for every nonzero vector in $\vspan\{n_{\mathcal{P}}\}$. This is precisely the fold non-degeneracy condition.\qedhere
\end{proof}

Thus, the previous proposition and Theorem~\ref{thm:dcurvesrefined-genexample} yield a finite-witness conclusion without requiring the rectifiable set to avoid the affine span of the pins. We state this formally in the following corollary.

\begin{corollary}[Pinned-Distance Finite Witnesses]
\label{cor:pinneddistancefinitewitness-examples}
Let $E\subseteq\mathbb{R}^d$ be a Borel $1$-rectifiable set satisfying $\mathcal{H}^1(E)>0$, and let $\mathcal{P}=\{p_1,\ldots,p_d\}\subseteq\mathbb{R}^d$ be affinely independent. Then there exists $j_*\in\{1,\ldots,d\}$ such that
$$
\mathcal{H}^1\big(\Delta_{p_{j_*}}(E)\big)>0.
$$
Moreover, the set
$$
\mathcal{B}_E
:=
\big\{p\in\mathbb{R}^d:
\mathcal{H}^1(\Delta_p(E))=0\big\}
$$
is contained in an affine subspace of dimension at most $d-2$.
\end{corollary}

\begin{proof}
Theorem~\ref{thm:dcurvesrefined-genexample} and Proposition~\ref{prop:pinneddistanceadmissible-examples} show that, for some $j_*$,
$$
\mathcal{H}^1\big(
\{|x-p_{j_*}|^2:x\in E\}
\big)>0.
$$
The functions $t\mapsto t^2$ and $t\mapsto\sqrt{t}$ are bilipschitz after restriction to compact subintervals of $(0,\infty)$. A countable localization therefore shows that positivity of length for a distance set is equivalent to positivity of length for the corresponding squared-distance set. This proves the finite-witness conclusion.

Every affinely independent $d$-tuple of pins is admissible by Proposition~\ref{prop:pinneddistanceadmissible-examples}. The exceptional-set conclusion now follows from Proposition~\ref{prop:parameterizedexceptionalset-genexample}.\qedhere
\end{proof}

The fold calculation also makes the finite-witness conclusion an open condition in the pin variables $\mathcal{P}$. In particular, if we let
$$
\mathfrak{F}_{\Delta}
:=
\big\{
(p_1,\ldots,p_d)\in(\mathbb{R}^d)^d:
p_1,\ldots,p_d\text{ are affinely independent}
\big\},
$$
then Proposition~\ref{prop:pinneddistanceadmissible-examples} shows that every tuple in $\mathfrak{F}_{\Delta}$ is globally fold non-degenerate.

\begin{corollary}[Stable Pinned-Distance Witnesses]
\label{cor:stablepinneddistancewitnesses-examples}
Fix an affinely independent tuple 
$$
\boldsymbol{p}^0=(p_1^0,\ldots,p_d^0).
$$
There then exist open neighbourhoods $p_j^0\in U_j\subseteq\mathbb{R}^d$ such that every tuple
$$
(p_1,\ldots,p_d)\in U_1\times\cdots\times U_d
$$
is affinely independent and its squared-distance encoding is globally fold non-degenerate. Moreover, for every Borel $1$-rectifiable set $E\subseteq\mathbb{R}^d$ satisfying $\mathcal{H}^1(E)>0$, there is an index $j_E\in\{1,\ldots,d\}$ for which
$$
\mathcal{H}^1\big(\Delta_p(E)\big)>0
\qquad
\text{for every }p\in U_{j_E}.
$$
The neighbourhoods depend only on $\boldsymbol{p}^0$, while the successful index may depend on $E$.
\end{corollary}

\begin{proof}
Choose the product neighbourhoods so that
$$
U_1\times\cdots\times U_d
\subseteq
\mathfrak{F}_{\Delta}.
$$
Global fold non-degeneracy throughout this product follows from Proposition~\ref{prop:pinneddistanceadmissible-examples}. Corollary~\ref{cor:globalstablegoodwitnessneighbourhoods-genexample}, followed by the countable bilipschitz localization between distances and squared distances used above, gives the conclusion.\qedhere
\end{proof}

\section{Planar radial projections}
\label{sec:radialprojections-examples}

Within the framework introduced here, the radial-projection family provides the sharp contrast with pinned distances: planar radial angle maps satisfy the critical-hypersurface condition, but fail both tangential immersion and fold non-degeneracy along the critical line determined by their vantage points.

Let $q_1,q_2\in\mathbb{R}^2$ be distinct vantage points and write
$$
L_{\mathcal{Q}}
:=
\Aff\{q_1,q_2\}.
$$
The radial projection based at $q_j$ is
$$
\pi_{q_j}(x)
:=
\frac{x-q_j}{|x-q_j|},
\qquad \forall x \in \mathbb{R}^2 \setminus \{q_j\}.
$$
Near any point different from $q_1$ and $q_2$, we may choose smooth angle lifts $\varphi_1$ and $\varphi_2$ satisfying
$$
\pi_{q_j}(x)
=
\big(\cos\varphi_j(x),\sin\varphi_j(x)\big).
$$
On any common open domain $U$ for such lifts, define
$$
\mathsf{H}_{\mathcal{Q}}(x)
:=
\big(\varphi_1(x),\varphi_2(x)\big),
$$
where $\varphi_j$ is thus the \textit{planar radial angle map} associated to the vantage point $q_j$.

We then have the following result detailing the underlying geometry of radial angle maps (and, by immediate consequence, radial projections).

\begin{proposition}[Critical Geometry of Planar Radial Projections]
\label{prop:radialprojectioncriticalgeometry-examples}
On every common angle-chart domain $U\subseteq\mathbb{R}^2\setminus\{q_1,q_2\}$, one has
$$
\Sigma_{\mathcal{Q}}\cap U
=
L_{\mathcal{Q}}\cap U.
$$
Thus the pair of radial angle maps satisfies the critical-hypersurface condition locally. At every $x\in L_{\mathcal{Q}}\cap U$, however,
$$
\rank\big(
D_x\mathsf{H}_{\mathcal{Q}}(x)
\vert_{T_xL_{\mathcal{Q}}}
\big)
=0.
$$
Consequently, the tangential-immersion condition fails along the critical line. The encoding is also not fold non-degenerate there.
\end{proposition}

\begin{proof}
Let $R(a,b):=(-b,a)$ denote counterclockwise rotation through $\pi/2$. A direct calculation gives that 
$$
\nabla\varphi_j(x)
=
\frac{R(x-q_j)}{|x-q_j|^2},
\qquad j=1,2.
$$
It follows that
\begin{equation}\label{eq:radialprojectionjacobian-examples}
J_{\mathcal{Q}}(x)
:=
\det D_x\mathsf{H}_{\mathcal{Q}}(x)
=
\frac{\det(x-q_1,x-q_2)}
{|x-q_1|^2|x-q_2|^2},
\end{equation}
up to a harmless choice of orientation. The numerator vanishes exactly when $x,q_1,q_2$ are collinear. Hence the critical set in $U$ is $L_{\mathcal{Q}}\cap U$.

If $x\in L_{\mathcal{Q}}\cap U$, then both vectors $x-q_j$ are parallel to $T_xL_{\mathcal{Q}}$, so their rotations are normal to the line. Thus each differential $D_x\varphi_j(x)$ vanishes on $T_xL_{\mathcal{Q}}$, which means that
$$
D_x\mathsf{H}_{\mathcal{Q}}(x)
\vert_{T_xL_{\mathcal{Q}}}=0.
$$
This is precisely the failure of tangential immersion.

At such a point the matrix $D_x\mathsf{H}_{\mathcal{Q}}(x)$ has rank one and its kernel is exactly $T_xL_{\mathcal{Q}}$. Since $J_{\mathcal{Q}}$ vanishes identically along the line,
$$
D_xJ_{\mathcal{Q}}(x)[v]=0,
\qquad
\forall v\in\ker D_x\mathsf{H}_{\mathcal{Q}}(x),
$$
which is precisely the negation of the fold non-degeneracy condition. \qedhere
\end{proof}

Although fold non-degeneracy fails on the critical line, the noncritical radial-projection conclusion remains open on compact sets disjoint from this closed line and hence at positive distance from it.

\begin{proposition}[Stable Radial Witnesses Away from the Critical Line]
\label{prop:stableradialwitnesses-examples}
Fix distinct points $q_1^0,q_2^0\in\mathbb{R}^2$ and a compact set
$$
K\subseteq
\mathbb{R}^2\setminus\Aff\{q_1^0,q_2^0\}.
$$
There exist open neighbourhoods $q_j^0\in U_j\subseteq\mathbb{R}^2$, $j=1,2$, such that
$$
K\cap\Aff\{q_1,q_2\}=\emptyset
\qquad
\text{for every }(q_1,q_2)\in U_1\times U_2.
$$
Moreover, for every Borel $1$-rectifiable set $E\subseteq K$ satisfying $\mathcal{H}^1(E)>0$, there is an index $j_E\in\{1,2\}$ for which
$$
\mathcal{H}^1\big(\pi_q(E)\big)>0
\qquad
\text{for every }q\in U_{j_E}.
$$
\end{proposition}

\begin{proof}
The determinant numerator in \eqref{eq:radialprojectionjacobian-examples} depends continuously on $(q_1,q_2,x)$ and is nonzero on the compact set $\{(q_1^0,q_2^0)\}\times K$. After shrinking product neighbourhoods $U_1\times U_2$, it remains nonzero for every $(q_1,q_2,x)\in U_1\times U_2\times K$. Thus every corresponding encoding has no critical point in $K$.

For each pair in this product, a finite localization to common angle charts followed by Corollary~\ref{cor:dcurves-genexample} shows that at least one of its two radial maps sends $E$ to a set of positive length.

Let $\mathcal{B}_E$ denote the bad vantage points for $E$. We have therefore proved
$$
(U_1\cap\mathcal{B}_E)
\times
(U_2\cap\mathcal{B}_E)
=
\emptyset.
$$
At least one factor must be empty, which proves the asserted stable good-witness conclusion.\qedhere
\end{proof}

As was previously observed in \cite{OS2011} and also in \cite{BBMTCompanion}, the failure of the good witness property is geometrically sharp. Indeed, if a positive-length rectifiable set $E$ is contained in $L_{\mathcal{Q}}$, then each of $\pi_{q_1}(E\setminus\{q_1\})$ and $\pi_{q_2}(E\setminus\{q_2\})$ contains at most two points. Thus both radial images have zero length. On the other hand, if
$$
\mathcal{H}^1(E\setminus L_{\mathcal{Q}})>0,
$$
then a countable localization to common angle charts, followed by Corollary~\ref{cor:dcurves-genexample}, shows that at least one of the two radial images has positive length. See also the work of Orponen and Sahlsten on radial projections of rectifiable sets \cite{OS2011}. 

In summary, the critical-hypersurface condition identifies the only possible obstruction, while the failure of tangential immersion explains why that obstruction is naturally and intrinsically excluded by the non-flatness hypothesis on $E$. It also explains the precise limit of parameter-space openness. Moreover, Proposition~\ref{prop:stableradialwitnesses-examples} is uniform only while the compact spatial set stays away from the moving critical line. No global analogue of Corollary~\ref{cor:stablepinneddistancewitnesses-examples} can hold for a radial pair: for each tuple $(q_1,q_2)$, a positive-length subset of $\Aff\{q_1,q_2\}$ has zero-length radial image from both vantage points.

\section{Generalized curve projections: framework and proofs}\label{sec:generalizedcurveprojections-genexample}

The discussion now turns from the specific pinned-distance and radial-projection models to \emph{generalized curve projections}. These are \emph{families} of mappings $\mathsf{\Phi} := \{\varphi_{\alpha}\}_{\alpha \in \mathcal{A}}$, which are indexed by a common topological space $\mathcal{A}$ (which is often assumed to be an open subset of real Euclidean space). As in earlier sections, a family of $C^2$ scalar mappings $\varphi_1,\ldots,\varphi_d : \mathbb{R}^d \rightarrow \mathbb{R}$ is again combined into a single canonical encoding map $\mathsf{H} : \mathbb{R}^d \rightarrow \mathbb{R}^d$. At a positive distance from the set of critical points of $\mathsf{H}$, the Manifold Inverse Function Theorem (which we have stated as Proposition~\ref{prop:manifoldinversefunction-diffgeo} in the Appendix) provides a locally bilipschitz change of variables, after which Federer's Projection Theorem applies.

\subsection{A variant of Federer's Theorem for generalized curve projections}\label{subsec:regularcurveprojections-genexample}

Fix a family
$$
\mathsf{\Phi} := \{\varphi_1,\ldots,\varphi_d\}
$$
of $C^2$ mappings $\varphi_j : \mathbb{R}^d \rightarrow \mathbb{R}$, whose associated canonical encoding map $\mathsf{H} :=\mathsf{H}_{\mathsf{\Phi}} : \mathbb{R}^d \rightarrow \mathbb{R}^d$ is given by
$$
\mathsf{H}_{\mathsf{\Phi}}(x)
:=
\big(\varphi_1(x),\ldots,\varphi_d(x)\big),
\quad x \in \mathbb{R}^d.
$$
We define its set of critical points by
$$
\Sigma_{\mathsf{\Phi}}
:=
\big\{x \in \mathbb{R}^d : \rank D_x\mathsf{H}_{\mathsf{\Phi}}(x) \leq d-1\big\}.
$$

\begin{remark}
Although the results are stated for mappings defined on all of $\mathbb{R}^d$, each argument is local and remains valid, with no change in the proof, for mappings $\varphi_j : \Omega \rightarrow \mathbb{R}$ defined on a common non-empty and open subset $\Omega \subseteq \mathbb{R}^d$. This observation will be used later when generalized distance functions are discussed.
\end{remark}

The first result records the compact localization principle underlying the arguments from the previous sections.

\begin{theorem}[A Variant of Federer's Theorem for Curve Projections]\label{thm:curvedfed-genexample}
Let 
$$
\mathsf{\Phi}:=\{\varphi_1,\ldots,\varphi_d\}
$$
be a family of $C^2$ mappings $\varphi_j : \mathbb{R}^d \rightarrow \mathbb{R}$ with associated canonical encoding map 
$$
\mathsf{H}_{\mathsf{\Phi}}(x):=\big(\varphi_1(x),\ldots,\varphi_d(x)\big),
\qquad x\in\mathbb{R}^d.
$$
Suppose that $E \subset \mathbb{R}^d$ is a Borel $1$-rectifiable set contained in a compact set $K \subset \mathbb{R}^d$ satisfying
$$
K \cap \Sigma_{\mathsf{\Phi}} = \emptyset.
$$
Then there exists a constant $C_{\mathsf{\Phi},K}>0$ such that
\begin{equation}\label{eq:curvedfedererinequality-genexample}
\mathcal{H}^1(E)
\leq
C_{\mathsf{\Phi},K}
\sum_{j=1}^{d}
\int_{\mathbb{R}}
\mathcal{H}^0\big(\varphi_j^{-1}(t) \cap E\big)
\,d\mathcal{H}^1(t).
\end{equation}
\end{theorem}

\begin{proof}
The critical set $\Sigma_{\mathsf{\Phi}}$ is closed, and $D_x\mathsf{H}_{\mathsf{\Phi}}(p)$ is invertible at every $p\in K$. The Inverse Function Theorem, in the form recorded in Proposition~\ref{prop:manifoldinversefunction-diffgeo}, therefore gives a sufficiently small open neighbourhood $p\in V_p\subset\mathbb{R}^d$ such that
$$
\mathsf{H}_{\mathsf{\Phi}}\big\vert_{V_p}
:
V_p \rightarrow \mathsf{H}_{\mathsf{\Phi}}(V_p)
$$
is bilipschitz.

By the compactness of $K$, we may choose a finite subcover $V_1,\ldots,V_N$ of such neighbourhoods, and then form a pairwise-disjoint partition of $E$ subordinate to this cover by setting
$$
E_1 := E \cap V_1,
\qquad
E_m := E \cap \bigg(V_m \setminus \bigcup_{\ell=1}^{m-1}V_\ell\bigg),
\quad 2 \leq m \leq N.
$$
For each $m$, let
$$
A_m := \mathsf{H}_{\mathsf{\Phi}}(E_m)
$$
and let
$$
L_m
:=
\operatorname{Lip}\bigg(
\big(\mathsf{H}_{\mathsf{\Phi}}\vert_{V_m}\big)^{-1}
\bigg).
$$
Since $\mathsf{H}_{\mathsf{\Phi}}\vert_{V_m}$ is bilipschitz, the set $A_m$ is $1$-rectifiable and
$$
\mathcal{H}^1(E_m) \leq L_m\mathcal{H}^1(A_m).
$$
We now apply Theorem~\ref{thm:federerprojection-intro} to $A_m$. The pointwise identity
$$
\pi_j\big(\mathsf{H}_{\mathsf{\Phi}}(x)\big)=\varphi_j(x)
$$
and the injectivity of $\mathsf{H}_{\mathsf{\Phi}}\vert_{V_m}$ give
\begin{equation}\label{eq:pointwisecountcorrespondence-genexample}
\mathcal{H}^0\big(\pi_j^{-1}(t) \cap A_m\big)
=
\mathcal{H}^0\big(\varphi_j^{-1}(t) \cap E_m\big)
\end{equation}
for every $t \in \mathbb{R}$ and every $j \in \{1,\ldots,d\}$. Consequently,
$$
\mathcal{H}^1(E_m)
\leq
L_m
\sum_{j=1}^{d}
\int_{\mathbb{R}}
\mathcal{H}^0\big(\varphi_j^{-1}(t) \cap E_m\big)
\,d\mathcal{H}^1(t).
$$
Since the sets $E_m$ are pairwise disjoint and $\mathcal{H}^0$ is counting measure, summing over the partition $E=\bigsqcup_{m=1}^{N}E_m$ proves \eqref{eq:curvedfedererinequality-genexample}, with
\[
C_{\mathsf{\Phi},K} := \max_{1\leq m\leq N}L_m.\qedhere
\]
\end{proof}

In exactly the same vein as the applications of Federer's Projection Theorem, the previous quantitative estimate has the following qualitative positive-image consequence.

\begin{corollary}[$d$-Curve Projections]\label{cor:dcurves-genexample}
Let $\mathsf{\Phi}:=\{\varphi_1,\ldots,\varphi_d\}$ and $\Sigma_{\mathsf{\Phi}}$ be as above. Suppose that $E \subset \mathbb{R}^d$ is a Borel $1$-rectifiable set satisfying
\begin{equation}\label{eq:rectifiableavoidssingular-genexample}
\mathcal{H}^1\big(E \setminus \Sigma_{\mathsf{\Phi}}\big)>0.
\end{equation}
Then there necessarily exists some $j^* \in \{1,\ldots,d\}$ such that
$$
\mathcal{H}^1\big(\varphi_{j^*}(E)\big)>0.
$$
\end{corollary}

\begin{proof}
For each $m,n \in \mathbb{N}$, define
$$
E^{(m,n)}
:=
E \cap \overline{B(0,m)}
\cap
\big\{x \in \mathbb{R}^d : \operatorname{dist}(x,\Sigma_{\mathsf{\Phi}}) \geq 1/n\big\}.
$$
Since $\Sigma_{\mathsf{\Phi}}$ is closed,
$$
E \setminus \Sigma_{\mathsf{\Phi}}
=
\bigcup_{m,n\in\mathbb{N}}E^{(m,n)}.
$$
It follows from \eqref{eq:rectifiableavoidssingular-genexample} that
$$
\mathcal{H}^1\big(E^{(m_0,n_0)}\big)>0
$$
for at least one pair $(m_0,n_0)$. The set $E^{(m_0,n_0)}$ is contained in the compact set
$$
K_{m_0,n_0}
:=
\overline{B(0,m_0)}
\cap
\big\{x \in \mathbb{R}^d : \operatorname{dist}(x,\Sigma_{\mathsf{\Phi}}) \geq 1/n_0\big\},
$$
which is disjoint from $\Sigma_{\mathsf{\Phi}}$. Theorem~\ref{thm:curvedfed-genexample} therefore gives
$$
0
<
\mathcal{H}^1\big(E^{(m_0,n_0)}\big)
\leq
C_{\mathsf{\Phi},K_{m_0,n_0}}
\sum_{j=1}^{d}
\int_{\mathbb{R}}
\mathcal{H}^0\big(\varphi_j^{-1}(t) \cap E^{(m_0,n_0)}\big)
\,d\mathcal{H}^1(t).
$$
Hence, at least one of the integrals on the right-hand side is positive. Since the integrand associated with $\varphi_j$ vanishes outside $\varphi_j(E)$, there exists some $j^*$ such that
\[
\mathcal{H}^1\big(\varphi_{j^*}(E)\big)>0. \qedhere
\]
\end{proof}

When $\varphi_j:=\pi_j$, the canonical encoding map is the identity on $\mathbb{R}^d$, so $\Sigma_{\mathsf{\Phi}}=\emptyset$. For the nonlinear families considered here, the possible presence of a critical set is the new obstruction and explains why a condition such as \eqref{eq:rectifiableavoidssingular-genexample} is a natural first sufficient hypothesis.

\subsubsection{Applying the critical hypersurface and tangential immersion conditions}\label{subsec:tangentialimmersion-genexample}
Using the critical-hypersurface and tangential-immersion conditions, together with Corollary~\ref{cor:dcurves-genexample} as a base case for our argument, we are now in a position to prove Theorem~\ref{thm:dcurvesrefined-genexample}, which is the central new result of this article.

\begin{proof}[Proof of Theorem~\ref{thm:dcurvesrefined-genexample}]
If
$
\mathcal{H}^1\big(E \setminus \Sigma_{\mathsf{\Phi}}\big)>0,
$
then the result follows immediately from Corollary~\ref{cor:dcurves-genexample}. We may therefore assume that
$$
\mathcal{H}^1\big(E \cap \Sigma_{\mathsf{\Phi}}\big)>0.
$$
Let $M \subset \mathbb{R}^d$ be the $(d-1)$-dimensional $C^1$ submanifold which is provided by the critical-hypersurface condition. Observe, then, that we must have
$$
E \cap \Sigma_{\mathsf{\Phi}} \subseteq M.
$$
For each index set
$$
I=\{i_1,\ldots,i_{d-1}\}\subset\{1,\ldots,d\},
$$
define
$$
\mathsf{H}^I(x)
:=
\big(\varphi_{i_1}(x),\ldots,\varphi_{i_{d-1}}(x)\big)
$$
and let
$$
\widetilde{\mathsf{H}}^I
:=
\mathsf{H}^I\vert_M : M \rightarrow \mathbb{R}^{d-1}.
$$

We can then apply the tangential-immersion condition (or, more precisely, its equivalent form given by Proposition~\ref{prop:restrictedrankequivalences-diffgeo}) to guarantee that, at every point $p \in M$, at least one such index set $I_p$ satisfies
$$
\rank D\widetilde{\mathsf{H}}^{I_p}(p)=d-1.
$$
For each such indexing set $I$, define the relatively open subset
$$
U_I
:=
\big\{p \in M : \rank D\widetilde{\mathsf{H}}^I(p)=d-1\big\}.
$$
As there are finitely many choices of $I$, the sets $U_I$ form a finite open cover of $M$; hence, there must exist at least one fixed index set $I$ such that
$$
\mathcal{H}^1\big(E \cap \Sigma_{\mathsf{\Phi}} \cap U_I\big)>0.
$$

By the Manifold Inverse Function Theorem (see Proposition~\ref{prop:manifoldinversefunction-diffgeo}), every point $p \in U_I$ has a relatively open neighbourhood $p \in V_p \subset M$ on which
$$
\widetilde{\mathsf{H}}^I\vert_{V_p}
:
V_p \rightarrow \widetilde{\mathsf{H}}^I(V_p)
$$
is a $C^1$ diffeomorphism. After shrinking $V_p$ if necessary, this restricted function is also a bilipschitz map. Since $M$ is an embedded submanifold, it is second countable, and hence we may choose a countable collection of these neighbourhoods covering $U_I$. It follows that at least one such neighbourhood $V_{p_0}$ satisfies
$
\mathcal{H}^1(F)>0,
$
where we have
$$
F
:=
E \cap \Sigma_{\mathsf{\Phi}} \cap U_I \cap V_{p_0}.
$$
The set
$$
A
:=
\widetilde{\mathsf{H}}^I(F)
\subset \mathbb{R}^{d-1}
$$
is therefore a $1$-rectifiable set of positive $\mathcal{H}^1$ measure. If $d=2$, then $I=\{i_1\}$ and
$$
\mathcal{H}^1\big(\varphi_{i_1}(E)\big)
\geq
\mathcal{H}^1\big(\varphi_{i_1}(F)\big)
=
\mathcal{H}^1(A)
>0,
$$
so the conclusion follows directly. We may therefore assume that $d\geq3$. Applying Federer's Projection Theorem in $\mathbb{R}^{d-1}$, there exists an index $\ell \in \{1,\ldots,d-1\}$ such that
$$
\mathcal{H}^1\big(\pi_{\ell,d-1}(A)\big)>0,
$$
where $\pi_{\ell,d-1}$ denotes projection onto the $\ell$-th coordinate in $\mathbb{R}^{d-1}$. Since
$$
\pi_{\ell,d-1}\circ\widetilde{\mathsf{H}}^I
=
\varphi_{i_\ell}\vert_M,
$$
we conclude that
\[
\mathcal{H}^1\big(\varphi_{i_\ell}(E)\big)
\geq
\mathcal{H}^1\big(\varphi_{i_\ell}(F)\big)
=
\mathcal{H}^1\big(\pi_{\ell,d-1}(A)\big)
>0.\qedhere
\]
\end{proof}

Corollary~\ref{cor:dcurves-genexample} assumes that $E$ has positive length away from $\Sigma_{\mathsf{\Phi}}$. Theorem~\ref{thm:dcurvesrefined-genexample} treats the complementary case by imposing tangential immersion along an embedded hypersurface $M$ containing $\Sigma_{\mathsf{\Phi}}$. This is precisely the motivation for the critical-hypersurface and tangential-immersion conditions.

\subsection{Exceptional-set estimates for families of generalized curve projections}\label{subsec:parameterizedexceptionalsets-genexample}

Theorem~\ref{thm:dcurvesrefined-genexample} is a statement about a fixed collection $\mathsf{\Phi} := \{\varphi_1,\ldots,\varphi_d\}$ of $C^2$ mappings $\varphi_j : \mathbb{R}^d \rightarrow \mathbb{R}$. In many of the forthcoming applications we will consider, however, the mappings $\varphi$ are selected from a larger parameterized family indexed by a common topological space $\mathcal{A}$.

The pinned-distance problem is a key example: each mapping $d_p:\mathbb{R}^d\rightarrow\mathbb{R}$ is parameterized by a pin $p\in\mathcal{A}:=\mathbb{R}^d$. For a fixed $1$-rectifiable set, we already considered the set of \emph{bad witness pins}
$$
\mathcal{B}_E := \big\{p \in \mathbb{R}^d : \mathcal{H}^1 \big(\Delta_p(E)\big) = 0 \big\},
$$
and established the structural estimate
$$
\dim \Aff\big(\mathcal{B}_E\big) \leq d - 2.
$$
In particular, when $\mathcal{B}_E$ is non-empty, it is contained in the affine subspace $V:=\Aff(\mathcal{B}_E)$, whose dimension is strictly smaller than that of the ambient parameter space.

We now formulate the analogous bad-witness parameters for generalized curve projections indexed by a topological parameter space $\mathcal{A}$, and then state geometric results which make their exceptional nature precise.

We begin with the following notation and definition.

\begin{definition}[Parameterized Families and Admissible Tuples]
Let $\mathcal{A}$ be a topological parameter set and let $\Omega \subseteq \mathbb{R}^d$ be an open domain. Define
$$
\mathscr{C}_{\mathcal{A}} := \big\{\varphi_{\alpha} : \Omega \rightarrow \mathbb{R} \, \lvert \, \alpha \in \mathcal{A} \big\}
$$
to be a family of $C^2$ mappings defined on the common domain $\Omega$ and indexed by $\mathcal{A}$. Given a tuple
$
\boldsymbol{\alpha}:=(\alpha_1,\ldots,\alpha_d)\in\mathcal{A}^d,
$
we write
$$
\mathsf{\Phi}_{\boldsymbol{\alpha}}
:=
\{\varphi_{\alpha_1},\ldots,\varphi_{\alpha_d}\} \subseteq \mathscr{C}_{\mathcal{A}}
$$
and then adapt the notational conventions
$
\mathsf{H}_{\boldsymbol{\alpha}} := \mathsf{H}_{\mathsf{\Phi}_{\boldsymbol{\alpha}}}
$
and
$
\Sigma_{\boldsymbol{\alpha}} = \Sigma_{\mathsf{\Phi}_{\boldsymbol{\alpha}}}.
$

Throughout the remainder of this section, plain parameters
$
\alpha,\beta\in\mathcal{A}
$
always denote individual indices, whereas bold parameters
$$
\boldsymbol{\alpha}
=
(\alpha_1,\ldots,\alpha_d),
\qquad
\boldsymbol{\beta}
=
(\beta_1,\ldots,\beta_d)
\in\mathcal{A}^d
$$
always denote ordered $d$-tuples of parameters.

We say that $\boldsymbol{\alpha}\in\mathcal{A}^d$ is an \textbf{admissible tuple} if the associated family $\mathsf{\Phi}_{\boldsymbol{\alpha}}$ is admissible; that is, it satisfies both the critical-hypersurface condition and the tangential-immersion condition. We denote the set of admissible tuples by
$
\mathfrak{G}\subseteq\mathcal{A}^d.
$
\end{definition}

Our original good witness property (stated as Definition~\ref{def:goodwitnessproperty-intro}) concerns a fixed family defined on all of $\mathbb{R}^d$. The following corresponding terminology keeps careful track of the quantifiers for parameterized families and compact localizations.

\begin{definition}[Good Witness Parameters and Localized Good-Witness Tuples]
\label{def:goodwitnesstuples-genexample}
Let $\Omega\subseteq\mathbb{R}^d$ be a common open domain for the family $\mathscr{C}_{\mathcal{A}}$, and let $K\subseteq\Omega$ be compact.

\begin{enumerate}
\item If $E\subseteq\Omega$ is a Borel $1$-rectifiable set satisfying $\mathcal{H}^1(E)>0$, we say that an individual parameter $\alpha\in\mathcal{A}$ is a \textbf{good witness for $E$} if
$$
\mathcal{H}^1\big(\varphi_\alpha(E)\big)>0.
$$

\item We say that a tuple
$$
\boldsymbol{\alpha}
=
(\alpha_1,\ldots,\alpha_d)
\in\mathcal{A}^d
$$
is a \textbf{good-witness tuple relative to $K$} if, for every Borel $1$-rectifiable set $E\subseteq K$ satisfying $\mathcal{H}^1(E)>0$, there exists an index $j_E\in\{1,\ldots,d\}$ such that
$$
\mathcal{H}^1\big(\varphi_{\alpha_{j_E}}(E)\big)>0.
$$
\end{enumerate}
\end{definition}

\begin{remark}
The order of quantifiers in the second definition is critical:
$$
\forall E\subseteq K, \, E \text{ is }1\text{-rectifiable}, \,
\exists j_E\in\{1,\ldots,d\} \text{ such that } \mathcal{H}^1 \big(\varphi_{\alpha_{j_E}} (E) \big) > 0.
$$
Hence, the successful coordinate is permitted to depend on $E$. Requiring one fixed coordinate $\alpha_j$ to witness every positive-length rectifiable subset of $K$ would be substantially stronger and is generally impossible, since a scalar mapping may be constant along positive-length portions of its level sets.

Thus, a good-witness tuple is a property of the tuple relative to the compact localization $K$, whereas the successful individual good witness may depend on the particular rectifiable set $E$. By Theorem~\ref{thm:dcurvesrefined-genexample}, every globally admissible tuple is a good-witness tuple relative to every compact set $K\subseteq\mathbb{R}^d$.
\end{remark}

The following main exceptional-set estimate is the abstract counterpart of the pinned-distance conclusion; Proposition~\ref{prop:pinneddistanceadmissible-examples} shows directly that the pinned squared-distance family satisfies its hypotheses.

\begin{proposition}[Exceptional Sets for Parameterized Families]\label{prop:parameterizedexceptionalset-genexample}
Let $E\subseteq\mathbb{R}^d$ be a Borel $1$-rectifiable set satisfying $\mathcal{H}^1(E)>0$. Let $\mathscr{C}_{\mathcal{A}}$ be a family of $C^2$ mappings $\varphi_\alpha:\mathbb{R}^d\to\mathbb{R}$ indexed by a topological space $\mathcal{A}$. Define the set of bad parameters by
$$
\mathcal{B}_E
:=
\big\{
\alpha\in\mathcal{A}:
\mathcal{H}^1\big(\varphi_\alpha(E)\big)=0
\big\},
$$
and let $\mathfrak{G}\subseteq\mathcal{A}^d$ be the associated set of admissible tuples.

\begin{enumerate}
\item{If $\mathcal{B}_E^d$ denotes the $d$-fold Cartesian product of $\mathcal{B}_E$, then
\begin{equation}\label{eq:badparametersavoidadmissibletuples-genexample}
\mathcal{B}_E^d\cap\mathfrak{G}=\emptyset.
\end{equation}
}

\item{As a special case, if $\mathcal{A}\subseteq\mathbb{R}^m$ and every affinely independent tuple $\alpha_1,\ldots,\alpha_d\in\mathcal{A}$ is admissible, then there exists an affine subspace $V_E\subseteq\mathbb{R}^m$ such that
$
\mathcal{B}_E\subseteq V_E
$
and
$
\dim V_E\leq d-2.
$
}
\end{enumerate}
\end{proposition}

Although this proposition is essentially a rephrasing of the preceding finite-witness result, we record it to motivate the openness results for good-witness tuples which immediately follow.

\begin{proof}
For a contradiction, suppose that
$$
\boldsymbol{\alpha}
=
(\alpha_1,\ldots,\alpha_d)
\in
\mathcal{B}_E^d\cap\mathfrak{G}.
$$
Then, as $\boldsymbol{\alpha}\in\mathcal{B}_E^d$, the set $\varphi_{\alpha_j}(E)$ is $\mathcal{H}^1$-null for each $j\in\{1,\ldots,d\}$. Since $\boldsymbol{\alpha}\in\mathfrak{G}$, Theorem~\ref{thm:dcurvesrefined-genexample} guarantees that at least one such image has positive $\mathcal{H}^1$ measure. This proves \eqref{eq:badparametersavoidadmissibletuples-genexample}.

Now suppose that $\mathcal{A}\subseteq\mathbb{R}^m$ and that every affinely independent $d$-tuple is admissible. For a contradiction, suppose that $\mathcal{B}_E$ is not contained in an affine subspace of dimension at most $d-2$, so that
$$
\dim\Aff(\mathcal{B}_E)\geq d-1.
$$
We may thus choose $d$ affinely independent points $\alpha_1^*,\ldots,\alpha_d^*\in\mathcal{B}_E$. The resulting tuple
$$
\boldsymbol{\alpha}^*
:=
(\alpha_1^*,\ldots,\alpha_d^*)
$$
is admissible by hypothesis, while also belonging to $\mathcal{B}_E^d$. This again contradicts \eqref{eq:badparametersavoidadmissibletuples-genexample}. Consequently,
$$
\dim\Aff(\mathcal{B}_E)\leq d-2,
$$
whenever $\mathcal{B}_E$ is non-empty, and we take $V_E:=\Aff(\mathcal{B}_E)$. If $\mathcal{B}_E$ is empty, any point of $\mathbb{R}^m$ may be used for $V_E$.\qedhere
\end{proof}

Using the previous result, an openness property for good-witness parameters follows.

\begin{theorem}[Local Abundance of Good Parameters]\label{thm:localgoodparameterneighbourhood-genexample}
Assume that $\mathcal{A}$ is a topological space and that $\mathfrak{G}\subseteq\mathcal{A}^d$ is open. Let $\boldsymbol{\alpha}^0=(\alpha_1^0,\ldots,\alpha_d^0)\in\mathfrak{G}$. Then there exist neighbourhoods $\alpha_j^0\in U_j\subseteq\mathcal{A}$ such that, for every Borel $1$-rectifiable set $E\subseteq\mathbb{R}^d$ satisfying $\mathcal{H}^1(E)>0$, there exists at least one index $j_E \in \{1,\ldots,d\}$ for which
$$
U_{j_E}\cap\mathcal{B}_E=\emptyset.
$$
In particular, every parameter in $U_{j_E}$ is good for $E$.
\end{theorem}

\begin{proof}
Since $\mathfrak{G}$ is open, we may choose neighbourhoods $U_1,\ldots,U_d$ such that
$$
U_1\times\cdots\times U_d\subseteq\mathfrak{G}.
$$
If every $U_j$ met $\mathcal{B}_E$, we could choose $\beta_j\in U_j\cap\mathcal{B}_E$ for every $j$. The tuple $(\beta_1,\ldots,\beta_d)$ would then belong both to $\mathfrak{G}$ and to $\mathcal{B}_E^d$, contradicting Proposition~\ref{prop:parameterizedexceptionalset-genexample}.\qedhere
\end{proof}

The strength of Theorem~\ref{thm:localgoodparameterneighbourhood-genexample} is its abstract generality: once the openness of $\mathfrak{G}$ is known, no differentiable structure on $\mathcal{A}$ is required, and the resulting neighbourhoods are uniform over all positive-length rectifiable sets $E\subseteq\mathbb{R}^d$, with only the successful index depending on $E$. Its limitation is that the openness of $\mathfrak{G}$ is assumed rather than established. Since membership in $\mathfrak{G}$ requires the global critical-hypersurface and tangential-immersion conditions, this assumption may be difficult to verify and need not persist under parameter perturbations. Thus the theorem gives a topological consequence of stability but supplies no mechanism for proving that stability. Fold non-degeneracy provides such a mechanism on compact spatial regions: it is a verifiable differential condition under which the critical hypersurface and its full tangential rank persist for nearby tuples on a prescribed compact set. This yields product neighbourhoods uniform over all rectifiable sets $E\subseteq K$ in Theorem~\ref{thm:stablegoodwitnessneighbourhoods-genexample}. It does not, by itself, prove that the globally admissible locus $\mathfrak{G}$ is open; the global conclusion is recovered only under the separate global fold-stability hypothesis of Corollary~\ref{cor:globalstablegoodwitnessneighbourhoods-genexample}.

\subsection{Fold non-degeneracy and local stability}\label{subsec:foldstability-genexample}

The critical-hypersurface and tangential-immersion conditions alone \emph{do not guarantee stability under parameter perturbation}, as the example later in this section shows: the determinant of the encoding map may vanish to higher order along the critical set. Fold non-degeneracy rules out this behaviour. Definition~\ref{def:foldnondegeneracy-genexample} gives the precise condition; this subsection develops its consequences for localized admissibility and parameter-space stability.

We first prove that the fold non-degeneracy condition implies the critical-hypersurface and tangential-immersion conditions.

\begin{lemma}[Fold Non-Degeneracy Implies Admissibility]\label{lem:foldimpliestangentialimmersion-genexample}
Suppose that $\mathsf{H}_{\boldsymbol{\alpha}}$ is fold non-degenerate on an open set $U\subseteq\mathbb{R}^d$. We then have the following conclusions.

\begin{enumerate}
\item The set $\Sigma_{\boldsymbol{\alpha}}\cap U$ is a $C^1$ embedded hypersurface.
\medskip
\item Every point $x \in \Sigma_{\boldsymbol{\alpha}} \cap U$ satisfies
$$
\rank\big(D_x\mathsf{H}_{\boldsymbol{\alpha}}(x)\vert_{T_x\Sigma_{\boldsymbol{\alpha}}}\big)=d-1.
$$
\end{enumerate}
In particular, the family $\mathsf{\Phi}_{\boldsymbol{\alpha}} \vert_U$ satisfies the critical-hypersurface and tangential-immersion conditions on $U$.
\end{lemma}

\begin{proof}
Fix $x\in\Sigma_{\boldsymbol{\alpha}}\cap U$ and choose $v\in\ker D_x\mathsf{H}_{\boldsymbol{\alpha}}(x)\setminus \{0\}$. Since $D_xJ_{\boldsymbol{\alpha}}(x)[v]\neq0$, the total derivative $D_xJ_{\boldsymbol{\alpha}}(x)$ cannot itself be zero. The Regular Level Set Theorem, in the form recalled in Subsection \ref{subsec:hypersurfacetransversality-diffgeo}, therefore shows that
$$
\Sigma_{\boldsymbol{\alpha}}\cap U
=
\big\{y \in U : \rank D_x\mathsf{H}_{\boldsymbol{\alpha}}(y) \leq d-1 \big\}
=
\big\{y\in U:J_{\boldsymbol{\alpha}}(y)=0\big\}
$$
is locally a $C^1$ hypersurface. Moreover, we know that
$$
T_x\Sigma_{\boldsymbol{\alpha}}=\ker D_xJ_{\boldsymbol{\alpha}}(x),
$$
which is identity \eqref{eq:regularleveltangent-diffgeo} from the appendix. Condition \eqref{eq:foldkerneltransversality-genexample} now gives $v\notin T_x\Sigma_{\boldsymbol{\alpha}}$. By \eqref{eq:foldrankstability-genexample}, the kernel of $D_x\mathsf{H}_{\boldsymbol{\alpha}}(x)$ is one-dimensional, so
$$
\ker D_x\mathsf{H}_{\boldsymbol{\alpha}}(x)
\cap
T_x\Sigma_{\boldsymbol{\alpha}}
=
\{0\}.
$$
Thus $D_x\mathsf{H}_{\boldsymbol{\alpha}}(x)$ is injective on the $(d-1)$-dimensional space $T_x\Sigma_{\boldsymbol{\alpha}}$, and its restriction has rank $d-1$, which follows from the kernel-transversality criterion from Proposition~\ref{prop:kerneltransversality-diffgeo}.\qedhere
\end{proof}

The preceding lemma shows that global fold non-degeneracy implies membership in $\mathfrak{G}$, while fold non-degeneracy on an open set gives the corresponding localized admissibility there. We next show that this localized condition is stable in parameter space after compact localization in the ambient spatial variable.

\begin{theorem}[Local Stability of Fold Non-Degeneracy]\label{thm:localfoldstability-genexample}
Let $\mathcal{A}\subseteq\mathbb{R}^m$ be open, and suppose that the mapping
$$
(\alpha,x)\longmapsto\varphi_\alpha(x), \qquad \forall (\alpha,x) \in \mathcal{A} \times \mathbb{R}^d
$$
is jointly $C^2$. Let $K\subset U\subseteq \mathbb{R}^d$, where $K$ is compact and $U$ is open. Fix $\boldsymbol{\alpha}^0 = (\alpha_1^0,\ldots,\alpha_d^0)\in\mathcal{A}^d$.

If $\mathsf{H}_{\boldsymbol{\alpha}^0}$ is fold non-degenerate at every critical point in $\overline{U}$, then the following hold simultaneously.

\begin{enumerate}
\item{There exists an open neighbourhood $\boldsymbol{\alpha}^0\in\mathcal{V}\subseteq\mathcal{A}^d$ such that, for every $\boldsymbol{\beta}\in\mathcal{V}$, the encoding map $\mathsf{H}_{\boldsymbol{\beta}}$ is fold non-degenerate at every critical point in $K$.}
\medskip
\item{Consequently, there exists an open set $W$ satisfying
$
K\subset W\Subset U
$
such that, for every $\boldsymbol{\beta} \in \mathcal{V}$, the associated family $\mathsf{\Phi}_{\boldsymbol{\beta}} \vert_W$ satisfies the critical-hypersurface and tangential-immersion conditions on $W$.}
\end{enumerate}
\end{theorem}

\begin{proof}
We again use the notation
$
J_{\boldsymbol{\beta}}(y)
:=
\det D_x\mathsf{H}_{\boldsymbol{\beta}}(y).
$
Since $K$ is compact and $U$ is open, we may first choose an open set $W\subseteq\mathbb{R}^d$ such that
$
K\subset W\Subset U.
$
We thus set
$
L:=\overline{W}
$
and
$$
Z
:=
\Sigma_{\boldsymbol{\alpha}^0}\cap L
=
\big\{
y\in L:
J_{\boldsymbol{\alpha}^0}(y)=0
\big\}.
$$
Since $L$ is compact and $J_{\boldsymbol{\alpha}^0}$ is continuous, the set $Z$ is compact.

If $Z$ is empty, then $|J_{\boldsymbol{\alpha}^0}|$ is bounded away from zero on $L$. The joint continuity of $(\boldsymbol{\beta},y)\mapsto J_{\boldsymbol{\beta}}(y)$ and the compactness of $L$ then give an open neighbourhood $\boldsymbol{\alpha}^0\in\mathcal{V}\subseteq\mathcal{A}^d$ for which $J_{\boldsymbol{\beta}}$ has no zeros on $L$ whenever $\boldsymbol{\beta}\in\mathcal{V}$. In this case $\mathsf{H}_{\boldsymbol{\beta}}$ has no critical points in $W$, so the fold non-degeneracy, critical-hypersurface, and tangential-immersion conclusions on $W$ are vacuous. We may therefore assume that $Z$ is non-empty.

Fix $x\in Z$. Using condition \eqref{eq:foldrankstability-genexample} of the fold non-degeneracy hypothesis, we know that
$$
\rank D_x\mathsf{H}_{\boldsymbol{\alpha}^0}(x)=d-1.
$$
After fixed permutations of the rows and columns, we may therefore assume that
$$
D_x\mathsf{H}_{\boldsymbol{\alpha}^0}(x)
=
\begin{pmatrix}
B_{\boldsymbol{\alpha}^0,x}
&
c_{\boldsymbol{\alpha}^0,x}
\\
r_{\boldsymbol{\alpha}^0,x}^{\mathsf{T}}
&
a_{\boldsymbol{\alpha}^0,x}
\end{pmatrix},
$$
where $B_{\boldsymbol{\alpha}^0,x}$ is an invertible $(d-1)\times(d-1)$ matrix. These permutations change the determinant by a fixed sign $\sigma_x\in\{-1,1\}$ but do not affect its zero set or the fold conditions. We then observe that the joint $C^2$ dependence of the mapping
$$
(\alpha,y)\longmapsto\varphi_\alpha(y),
\qquad
(\alpha,y)\in\mathcal{A}\times\mathbb{R}^d,
$$
implies that the mapping
$$
(\boldsymbol{\beta},y)
\longmapsto
D_x\mathsf{H}_{\boldsymbol{\beta}}(y)
$$
is continuous. Consequently, there are open neighbourhoods
$$
\boldsymbol{\alpha}^0\in\mathcal{V}_x\subseteq\mathcal{A}^d
\qquad\text{and}\qquad
x\in G_x\subseteq U
$$
such that
\begin{equation}\label{eq:localderivativestructurefold-genexample}
D_x\mathsf{H}_{\boldsymbol{\beta}}(y)
=
\begin{pmatrix}
B_{\boldsymbol{\beta},y}
&
c_{\boldsymbol{\beta},y}
\\
r_{\boldsymbol{\beta},y}^{\mathsf{T}}
&
a_{\boldsymbol{\beta},y}
\end{pmatrix},
\qquad
(\boldsymbol{\beta},y)
\in
\mathcal{V}_x\times G_x,
\end{equation}
and, by continuity, we may assume that the minor $B_{\boldsymbol{\beta},y}$ remains invertible throughout this product neighbourhood. On $\mathcal{V}_x\times G_x$, define the vector-valued function
$$
(\boldsymbol{\beta},y)
\longmapsto
v_{\boldsymbol{\beta},y}
\in\mathbb{R}^d
$$
by setting
$$
v_{\boldsymbol{\beta},y}
:=
\begin{pmatrix}
-B_{\boldsymbol{\beta},y}^{-1}c_{\boldsymbol{\beta},y}
\\
1
\end{pmatrix}.
$$
Recall that $B_{\boldsymbol{\beta},y}$ is the $(d-1)\times(d-1)$ invertible square matrix from \eqref{eq:localderivativestructurefold-genexample}, whereas $c_{\boldsymbol{\beta},y}$ is the $(d-1)\times1$ column vector from \eqref{eq:localderivativestructurefold-genexample}.

By the preceding discussion, $v_{\boldsymbol{\beta},y}$ is jointly continuous in $(\boldsymbol{\beta},y)\in\mathcal{V}_x\times G_x$. Moreover, it is nonzero because its final coordinate is equal to $1$. By construction, we know that
\begin{equation}\label{eq:totalderivativeflat-genexample}
D_x\mathsf{H}_{\boldsymbol{\beta}}(y)
v_{\boldsymbol{\beta},y}
=
\begin{pmatrix}
0
\\
a_{\boldsymbol{\beta},y}
-
r_{\boldsymbol{\beta},y}^{\mathsf{T}}
B_{\boldsymbol{\beta},y}^{-1}
c_{\boldsymbol{\beta},y}
\end{pmatrix},
\qquad
\forall (\boldsymbol{\beta},y)
\in
\mathcal{V}_x\times G_x,
\end{equation}
whereas the block determinant formula gives
\begin{equation}\label{eq:blockdeterminantflat-genexample}
J_{\boldsymbol{\beta}}(y)
=
\sigma_x\big(\det B_{\boldsymbol{\beta},y}\big)
\big(
a_{\boldsymbol{\beta},y}
-
r_{\boldsymbol{\beta},y}^{\mathsf{T}}
B_{\boldsymbol{\beta},y}^{-1}
c_{\boldsymbol{\beta},y}
\big).
\end{equation}
Comparing \eqref{eq:totalderivativeflat-genexample} and \eqref{eq:blockdeterminantflat-genexample} shows that whenever $(\boldsymbol{\beta},y)\in\mathcal{V}_x\times G_x$ and $J_{\boldsymbol{\beta}}(y)=0$, we have
$$
D_x\mathsf{H}_{\boldsymbol{\beta}}(y)
v_{\boldsymbol{\beta},y}=0_d.
$$
This follows because both quantities vanish if and only if the following scalar-valued identity is satisfied
$$
a_{\boldsymbol{\beta},y}
-
r_{\boldsymbol{\beta},y}^{\mathsf{T}}
B_{\boldsymbol{\beta},y}^{-1}
c_{\boldsymbol{\beta},y} = 0.
$$
Again, we have used the invertibility of the minor $B_{\boldsymbol{\beta},y}$ to establish this equivalence.

Since $B_{\boldsymbol{\beta},y}$ is invertible, the total derivative matrix $D_x\mathsf{H}_{\boldsymbol{\beta}}(y)$ has rank $ \geq d-1$. In the complementary direction, since its determinant vanishes at the critical point $y$, it also has rank at most $d-1$. Therefore,
$$
\rank D_x\mathsf{H}_{\boldsymbol{\beta}}(y)=d-1
\qquad\text{and hence}\qquad
\dim \ker D_x\mathsf{H}_{\boldsymbol{\beta}}(y) = 1.
$$
The preceding calculations thus demonstrate that
$$
\ker D_x\mathsf{H}_{\boldsymbol{\beta}}(y)
=
\operatorname{span}\{v_{\boldsymbol{\beta},y}\}.
$$
In particular, for every fixed $\boldsymbol{\beta}\in\mathcal{V}_x$, we have verified the first requirement of the fold non-degeneracy condition (which is \eqref{eq:foldrankstability-genexample} in Definition \ref{def:foldnondegeneracy-genexample}), and this verification holds for every critical point $y\in\Sigma_{\boldsymbol{\beta}}\cap G_x$.

Now, at the original base point $(\boldsymbol{\alpha}^0,x)\in\mathcal{V}_x\times G_x$, the assumption of fold non-degeneracy guarantees that
$$
D_xJ_{\boldsymbol{\alpha}^0}(x)
[v_{\boldsymbol{\alpha}^0,x}]
\neq0.
$$
Because the mapping $(\alpha,y)\mapsto\varphi_\alpha(y)$ is jointly $C^2$, the induced mapping
$$
(\boldsymbol{\beta},y)
\longmapsto
D_xJ_{\boldsymbol{\beta}}(y)
$$
is continuous; since the mapping $(\boldsymbol{\beta},y)\mapsto v_{\boldsymbol{\beta},y}$ is also continuous, the function
$$
(\boldsymbol{\beta},y)
\longmapsto
D_xJ_{\boldsymbol{\beta}}(y)
[v_{\boldsymbol{\beta},y}]
$$
is itself continuous. After perhaps shrinking the open component sets $\mathcal{V}_x$ and $G_x$, we can thus arrange that
\begin{equation}\label{eq:criticalnonvanishingfold-genexample}
D_xJ_{\boldsymbol{\beta}}(y)
[v_{\boldsymbol{\beta},y}]
\neq0
\end{equation}
throughout $\mathcal{V}_x\times G_x$.

Fix $\boldsymbol{\beta} \in \mathcal{V}_x$ and let $y\in\Sigma_{\boldsymbol{\beta}}\cap G_x$ be a critical point for the associated canonical encoding map $\mathsf{H}_{\boldsymbol{\beta}}$. Since
$$
\ker D_x\mathsf{H}_{\boldsymbol{\beta}}(y)
=
\operatorname{span}\{v_{\boldsymbol{\beta},y}\},
$$
every $w\in\ker D_x\mathsf{H}_{\boldsymbol{\beta}}(y)\setminus \{0\}$ has the form
$$
w=\lambda v_{\boldsymbol{\beta},y}
$$
for some $\lambda\neq0$. It follows from \eqref{eq:criticalnonvanishingfold-genexample} that
$$
D_xJ_{\boldsymbol{\beta}}(y)[w]
=
\lambda
D_xJ_{\boldsymbol{\beta}}(y)
[v_{\boldsymbol{\beta},y}]
\neq0.
$$
This verifies the second requirement of the fold non-degeneracy condition (which is \eqref{eq:foldkerneltransversality-genexample} in Definition \ref{def:foldnondegeneracy-genexample}). This verification again holds at every critical point $y\in\Sigma_{\boldsymbol{\beta}}\cap G_x$.

Now, the family of component neighbourhoods $G_x \subseteq U$, indexed by the base points $x\in Z$, covers the compact set $Z$. Choosing a finite subcover
$$
Z\subseteq G_{x_1}\cup\cdots\cup G_{x_N}
$$
and defining
$$
G:=G_{x_1}\cup\cdots\cup G_{x_N},
\qquad
\mathcal{V}_0
:=
\mathcal{V}_{x_1}\cap\cdots\cap\mathcal{V}_{x_N},
$$
we obtain that, for every $\boldsymbol{\beta}\in\mathcal{V}_0$, the canonical encoding map $\mathsf{H}_{\boldsymbol{\beta}}$ is fold non-degenerate at every point of $\Sigma_{\boldsymbol{\beta}}\cap G$.

It remains to rule out the appearance of new critical points in $L\setminus G$. Since every critical point of $\mathsf{H}_{\boldsymbol{\alpha}^0}$ in $L$ belongs to $Z\subseteq G$, the function $J_{\boldsymbol{\alpha}^0}$ has no zeros on the compact set $L\setminus G$. If $L\setminus G$ is non-empty, there exists a constant $\delta>0$ such that
$$
\big|J_{\boldsymbol{\alpha}^0}(y)\big|
\geq\delta,
\qquad
y\in L\setminus G.
$$
The continuity of $(\boldsymbol{\beta},y)\mapsto J_{\boldsymbol{\beta}}(y)$ and the compactness of $L$ then allow us to shrink to an open neighbourhood
$$
\boldsymbol{\alpha}^0
\in
\mathcal{V}
\subseteq
\mathcal{V}_0
$$
upon which
$$
\sup_{y\in L}
\big|
J_{\boldsymbol{\beta}}(y)
-
J_{\boldsymbol{\alpha}^0}(y)
\big|
<
\frac{\delta}{2},
\qquad
\boldsymbol{\beta}\in\mathcal{V}.
$$
The triangle inequality gives
$$
\big|J_{\boldsymbol{\beta}}(y)\big|
\geq
\frac{\delta}{2},
\qquad
y\in L\setminus G.
$$
Hence $\mathsf{H}_{\boldsymbol{\beta}}$ has no critical points in $L\setminus G$. If $L\setminus G$ is empty, we simply take $\mathcal{V}:=\mathcal{V}_0$.

We have thus shown that every critical point of $\mathsf{H}_{\boldsymbol{\beta}}$ in $L$ lies in $G$, where the corank-one condition and the kernel-transversality condition both hold. Hence, $\mathsf{H}_{\boldsymbol{\beta}}$ is fold non-degenerate at every critical point in $L$ for every $\boldsymbol{\beta}\in\mathcal{V}$. Since
$$
K\subset W\subseteq L,
$$
this proves the first conclusion and, in fact, proves fold non-degeneracy at every critical point in $W$. This establishes part (1) of Theorem~\ref{thm:localfoldstability-genexample}.

We then fix $\boldsymbol{\beta}\in\mathcal{V}$ and let $y\in W\cap\Sigma_{\boldsymbol{\beta}}$. We have already demonstrated that \eqref{eq:criticalnonvanishingfold-genexample} implies
$$
D_xJ_{\boldsymbol{\beta}}(y)\neq0.
$$
The Regular Level Set Theorem, as recalled in Subsection \ref{subsec:hypersurfacetransversality-diffgeo}, therefore shows that
$$
\Sigma_{\boldsymbol{\beta}}\cap W
=
\big\{
z\in W:
J_{\boldsymbol{\beta}}(z)=0
\big\}
$$
is a $C^1$ hypersurface, whenever it is non-empty. Its tangent space at each critical point $y\in W\cap\Sigma_{\boldsymbol{\beta}}$ satisfies the identity
$$
T_y\Sigma_{\boldsymbol{\beta}}
=
\ker D_xJ_{\boldsymbol{\beta}}(y).
$$
This is identity \eqref{eq:regularleveltangent-diffgeo} with $J=J_{\boldsymbol{\beta}}$. Moreover, the fold non-degeneracy condition implies
$$
\ker D_x\mathsf{H}_{\boldsymbol{\beta}}(y)
\cap
T_y\Sigma_{\boldsymbol{\beta}}
=
\{0\},
$$
and hence
$$
\rank\big(
D_x\mathsf{H}_{\boldsymbol{\beta}}(y)
\vert_{T_y\Sigma_{\boldsymbol{\beta}}}
\big)
=
d-1.
$$
The last implication is exactly Proposition~\ref{prop:kerneltransversality-diffgeo}. This is precisely the critical-hypersurface and tangential-immersion conclusion on the open set $W$. This establishes part (2) of Theorem~\ref{thm:localfoldstability-genexample}.\qedhere
\end{proof}

The preceding pointwise good-witness results for $\boldsymbol{\beta} \in \mathcal{A}^d$ yield the following local stability in parameter space.

\begin{theorem}[Stable Neighbourhoods of Good Witnesses]\label{thm:stablegoodwitnessneighbourhoods-genexample}
Assume the hypotheses of Theorem~\ref{thm:localfoldstability-genexample}. Then there exist open neighbourhoods
$$
\alpha_j^0\in U_j\subseteq\mathcal{A},
\qquad j=1,\ldots,d,
$$
such that the following holds. For every Borel $1$-rectifiable set $E\subseteq K$ satisfying $\mathcal{H}^1(E)>0$, there exists an index $j_E\in\{1,\ldots,d\}$ for which
\begin{equation}\label{eq:stablegoodwitnessneighbourhood-genexample}
\mathcal{H}^1\big(\varphi_\alpha(E)\big)>0
\qquad
\text{for every }\alpha\in U_{j_E}.
\end{equation}
The neighbourhoods $U_1,\ldots,U_d$ depend only on the tuple $\boldsymbol{\alpha}^0$ and the compact localization $K$, while the successful index $j_E$ may depend on $E$.
\end{theorem}

The order of quantifiers in Theorem~\ref{thm:stablegoodwitnessneighbourhoods-genexample} shows that the result is \emph{uniform over choices of $1$-rectifiable sets} $E \subset K \subset \mathbb{R}^d$, up to the choice of good-witness index $j_E \in \{1,\ldots,d\}$, under the neighbourhood fold hypothesis of Theorem~\ref{thm:localfoldstability-genexample}. In particular, every tuple in
$$
U_1\times\cdots\times U_d
$$
is a good-witness tuple relative to $K$ in the sense of Definition~\ref{def:goodwitnesstuples-genexample}. After the proof, we state and prove a similar openness result with associated neighbourhoods $U_1^E,\ldots,U_d^E \subset \mathcal{A}$ depending upon the choice of $1$-rectifiable set $E \subset \mathbb{R}^d$.

\begin{proof}
Let $\mathcal{V}\subseteq\mathcal{A}^d$ be the neighbourhood furnished by Theorem~\ref{thm:localfoldstability-genexample}. Since $\mathcal{V}$ is open, we may choose open neighbourhoods
$
U_1,\ldots,U_d  \subseteq \mathcal{A}$ such that 
$
\boldsymbol{\alpha}^0 =(\alpha_1^0,\ldots,\alpha_d^0) \in U_1 \times \cdots \times U_d
$
and also
$
U_1\times\cdots\times U_d \subseteq \mathcal{V}.
$
Fix any Borel $1$-rectifiable set $E\subseteq K$ of positive length and let
$
\boldsymbol{\beta} = (\beta_1,\ldots,\beta_d) \in U_1\times\cdots\times U_d
$
be an arbitrary $d$-tuple of local indices near $\boldsymbol{\alpha}^0$.

As an initial consideration, if
$
\mathcal{H}^1\big(E\setminus\Sigma_{\boldsymbol{\beta}}\big)>0,
$
then Corollary~\ref{cor:dcurves-genexample}, applied after an appropriate compact localization away from $\Sigma_{\boldsymbol{\beta}}$, shows that at least one of the images $\varphi_{\beta_j}(E)$ has positive length. 

Otherwise, we may freely assume that
$$
\mathcal{H}^1\big(E\cap K \cap \Sigma_{\boldsymbol{\beta}}\big)>0.
$$
At every critical point in $K$, Lemma~\ref{lem:foldimpliestangentialimmersion-genexample} gives a $C^1$ critical hypersurface and the rank identity
$$
\rank\big(
D_x\mathsf{H}_{\boldsymbol{\beta}}(x)
\vert_{T_x\Sigma_{\boldsymbol{\beta}}}
\big)
=
d-1.
$$
The proof of Theorem~\ref{thm:dcurvesrefined-genexample}, localized to a neighbourhood of the compact set $\Sigma_{\boldsymbol{\beta}}\cap K$, again shows that at least one $\varphi_{\beta_j}(E)$ has positive length. Consequently, no tuple in $U_1\times\cdots\times U_d$ can consist entirely of bad parameters for $E$.

Writing $\mathcal{B}_E$ for the bad-parameter set, we have proved
$$
\big(U_1\cap\mathcal{B}_E\big)
\times\cdots\times
\big(U_d\cap\mathcal{B}_E\big)
=
\emptyset.
$$
A finite Cartesian product is empty only if at least one factor is empty. Hence
$$
U_{j_E}\cap\mathcal{B}_E=\emptyset
$$
for some $j_E$, which is exactly \eqref{eq:stablegoodwitnessneighbourhood-genexample}.\qedhere
\end{proof}

The compact localization in Theorem~\ref{thm:stablegoodwitnessneighbourhoods-genexample} is essential: local fold non-degeneracy controls only those portions of a rectifiable set that remain in the region of stability. A conclusion uniform over all positive-length rectifiable sets in $\mathbb{R}^d$ is available when fold non-degeneracy persists globally on a full parameter neighbourhood.

\begin{corollary}[Global Stable Neighbourhoods of Good Witnesses]
\label{cor:globalstablegoodwitnessneighbourhoods-genexample}
Let $\mathcal{A}\subseteq\mathbb{R}^m$ be open, and suppose that $(\alpha,x)\mapsto\varphi_\alpha(x)$ is jointly $C^2$ on $\mathcal{A}\times\mathbb{R}^d$. Fix a tuple
$
\boldsymbol{\alpha}^0 = (\alpha_1^0,\ldots,\alpha_d^0) \in\mathcal{A}^d,
$
and suppose that there exists an open neighbourhood
$
\boldsymbol{\alpha}^0 \in \mathcal{V} \subseteq \mathcal{A}^d
$
such that, for every $\boldsymbol{\beta}\in\mathcal{V}$, the associated encoding
map $\mathsf{H}_{\boldsymbol{\beta}}$ is fold non-degenerate at all of its critical points $y \in \Sigma_{\boldsymbol{\beta}} \subset \mathbb{R}^d$. Then there exist open neighbourhoods
$$
\alpha_j^0\in U_j\subseteq\mathcal{A},
\qquad
j=1,\ldots,d,
$$
such that the following holds. For every Borel $1$-rectifiable set
$E\subseteq\mathbb{R}^d$ satisfying $\mathcal{H}^1(E)>0$, there exists
an index $j_E\in\{1,\ldots,d\}$ for which
\begin{equation}\label{eq:globalstablegoodwitness-genexample}
\mathcal{H}^1\big(\varphi_\alpha(E)\big)>0
\qquad
\text{for every }\alpha\in U_{j_E}.
\end{equation}
The neighbourhoods $U_1,\ldots,U_d$ depend only on
$\boldsymbol{\alpha}^0$ and the global fold-stability neighbourhood
$\mathcal{V}$, while the specific choice of good-witness index $j_E$ may depend on $E$.
\end{corollary}

\begin{proof}
Since $\mathcal{V}$ is open, we may choose open neighbourhoods
$\alpha_j^0\in U_j\subseteq\mathcal{A}$ such that
$$
U_1\times\cdots\times U_d
\subseteq
\mathcal{V}.
$$
Fix a Borel $1$-rectifiable set $E\subseteq\mathbb{R}^d$ satisfying
$\mathcal{H}^1(E)>0$, and let
$$
\mathcal{B}_E
:=
\left\{
\alpha\in\mathcal{A}:
\mathcal{H}^1\big(\varphi_\alpha(E)\big)=0
\right\}.
$$
Suppose, for the sake of contradiction, that every $U_j$ meets
$\mathcal{B}_E$. We could then choose
$$
\beta_j\in U_j\cap\mathcal{B}_E,
\qquad
j=1,\ldots,d.
$$
The resulting tuple
$$
\boldsymbol{\beta}
=
(\beta_1,\ldots,\beta_d)
\in
U_1\times\cdots\times U_d
\subseteq
\mathcal{V}
$$
is globally fold non-degenerate. Consequently, the associated family
$\mathsf{\Phi}_{\boldsymbol{\beta}}$ satisfies the
critical-hypersurface and tangential-immersion conditions on
$\mathbb{R}^d$. Theorem~\ref{thm:dcurvesrefined-genexample} therefore produces an index $j\in\{1,\ldots,d\}$ such that
$$
\mathcal{H}^1\big(\varphi_{\beta_j}(E)\big)>0.
$$
This contradicts $\beta_j\in\mathcal{B}_E$. Hence, one factor
$U_{j_E}\cap\mathcal{B}_E$ must be empty, which is
\eqref{eq:globalstablegoodwitness-genexample}.\qedhere
\end{proof}

For a fixed $1$-rectifiable set $E \subset \mathbb{R}^d$, the stability of a particular good witness requires less than fold non-degeneracy. We record this fact separately because it clarifies which part of Theorem~\ref{thm:stablegoodwitnessneighbourhoods-genexample} (namely, uniformity in the product components $U_1,\ldots,U_d \subseteq \mathcal{A}$ over all rectifiable $1$-sets $E$) comes from the fold hypothesis.

\begin{proposition}[Openness of the Good-Witness Condition]\label{prop:fixedgoodwitnessopen-genexample}
Let $K\subseteq\mathbb{R}^d$ be compact, let $\mathcal{A}\subseteq\mathbb{R}^m$ be open, and suppose that the mapping
$
(\alpha,x)\mapsto\varphi_{\alpha}(x)
$
is jointly $C^1$ on a neighbourhood of $\mathcal{A}\times K$. 

If $E\subseteq K$ is a Borel $1$-rectifiable set, define the set of $E$-dependent good-witness parameters
$$
\mathcal{G}_E
:=
\big\{
\alpha\in\mathcal{A}:
\mathcal{H}^1\big(\varphi_\alpha(E)\big)>0
\big\}.
$$
Then $\mathcal{G}_E$ is an open subset of $\mathcal{A}$.
\end{proposition}

Throughout the following proof, we use basic notions of rectifiability, such as approximate tangent lines, tangential Jacobians, and the Area Formula. A standard reference for this material is \cite[Chapter 3]{EG2025}.

\begin{proof}
Fix $\alpha_0\in\mathcal{G}_E$. By subdividing the Lipschitz parametrizations in the definition of rectifiability into compact intervals, there exist compact Lipschitz curve images $\Gamma_1,\Gamma_2,\ldots\subseteq\mathbb{R}^d$, each of finite $\mathcal{H}^1$ measure, and an $\mathcal{H}^1$-null set $N$ such that
$$
E\subseteq N\cup\bigcup_{n=1}^{\infty}\Gamma_n.
$$
The map $\varphi_{\alpha_0}$ is locally Lipschitz on a neighbourhood of the compact set $K$, and hence $\mathcal{H}^1(\varphi_{\alpha_0}(N))=0$. Since $\mathcal{H}^1(\varphi_{\alpha_0}(E))>0$, countable subadditivity therefore gives an index $n_0$ such that
$$
\mathcal{H}^1\big(
\varphi_{\alpha_0}(E\cap\Gamma_{n_0})
\big)>0.
$$
Replacing $E$ by $E\cap\Gamma_{n_0}$ for the remainder of the proof, we may assume that $E\subseteq\Gamma_{n_0}$ and $\mathcal{H}^1(E)<\infty$. Hence, at $\mathcal{H}^1$-almost every $x \in E$, there exists a unit vector $\tau_E(x) \in \mathbb{S}^{d-1}$ whose span is the approximate tangent line to $E$ at $x$. For each $\alpha \in \mathcal{A}$, the associated tangential Jacobian is then the function
$$
J_E\varphi_{\alpha}(x)
:=
\big|
D_x\varphi_{\alpha}(x)[\tau_E(x)]
\big|
$$
for $\mathcal{H}^1$-almost every $x\in E$.

By the Area Formula, we have
\begin{equation}\label{eq:areaformulagoodwitness-genexample}
\int_E J_E\varphi_{\alpha_0}(x)\,d\mathcal{H}^1(x)
=
\int_{\mathbb{R}}
N\big(\varphi_{\alpha_0}\vert_E,t\big)
\,d\mathcal{H}^1(t),
\end{equation}
where 
$$
N(\varphi_{\alpha_0}\vert_E,t) := \# \big(\varphi_{\alpha_0}^{-1} (t) \cap E \big), \qquad \forall t \in \mathbb{R}.
$$
Since $\mathcal{H}^1(\varphi_{\alpha_0}(E))>0$, the integral in \eqref{eq:areaformulagoodwitness-genexample} is positive. Hence, choose a measurable set $F\subseteq E$ and an associated constant $c>0$ such that
$
0<\mathcal{H}^1(F)<\infty
$
and also
$
J_E\varphi_{\alpha_0}(x)\geq c
$
for $\mathcal{H}^1$-almost every $x\in F$.

Now, the joint $C^1$-continuity of the map $(\alpha, x) \mapsto D_x\varphi_{\alpha}$ together with the compactness of $K$ thus give a neighbourhood $\alpha_0\in W\subseteq\mathcal{A}$ such that
$$
\sup_{x\in K}
\big \lvert
D_x\varphi_{\alpha}(x)-D_x\varphi_{\alpha_0}(x)
\big\rvert
<
\frac{c}{2}
$$
for every $\alpha\in W$. It follows that
$$
J_E\varphi_{\alpha}(x)\geq\frac{c}{2}
$$
for almost every $x\in F$. Another application of the Area Formula thus yields
$$
\int_{\mathbb{R}}
N\big(\varphi_{\alpha}\vert_F,t\big)
\,d\mathcal{H}^1(t)
\geq
\frac{c}{2}\mathcal{H}^1(F)
>
0, \qquad \forall \alpha \in W.
$$
Therefore,
$$
\mathcal{H}^1\big(\varphi_{\alpha}(E)\big) \geq \mathcal{H}^1\big(\varphi_{\alpha}(F)\big)>0,
\qquad
\forall\alpha\in W.
$$
This proves the result.\qedhere
\end{proof}

Corollary~\ref{cor:globalstablegoodwitnessneighbourhoods-genexample} and Proposition~\ref{prop:fixedgoodwitnessopen-genexample} thus provide a counterpoint to our discussion of openness conditions for good-witness parameters. If the corresponding openness is required to be \textit{uniform over all choices of rectifiable $1$-sets}, then an additional non-degeneracy condition on the family must be imposed. However, if the neighbourhood is allowed to vary with the choice of rectifiable $1$-set $E$, then standard arguments from classical geometric measure theory imply the result. In this sense, Proposition~\ref{prop:fixedgoodwitnessopen-genexample} highlights the strength of our uniform openness results obtained from fold non-degeneracy.

The following example is included to illustrate why one requires the compact localization hypothesis in Theorem \ref{thm:localfoldstability-genexample}.

\begin{example}[Global Instability and the Necessity of Compact Localization]
\label{ex:globalinstabilityescapingcriticalline-genexample}
For each $\varepsilon\in(-1,1)$, consider the family
$$
\mathsf{\Phi}_{\varepsilon}
:=
\big\{
\varphi_{1,\varepsilon},
\varphi_{2,\varepsilon}
\big\},
$$
where
$$
\varphi_{1,\varepsilon}(x,y):=x
$$
and
$$
\varphi_{2,\varepsilon}(x,y)
:=
(1-\varepsilon x)^2y.
$$
The associated canonical encoding map is
$$
\mathsf{H}_{\varepsilon}(x,y)
=
\big(
x,(1-\varepsilon x)^2y
\big).
$$
For any $\varepsilon \in (-1,1)$, the mapping
$$
(\varepsilon; x,y)
\longmapsto
\mathsf{H}_{\varepsilon}(x,y)
$$
is jointly $C^\infty$ on $(-1,1) \times \mathbb{R}^2$. Moreover,
$$
\mathsf{H}_0(x,y)=(x,y),
$$
so the base encoding map is the identity and has no critical points.
In particular, it is fold non-degenerate on all of $\mathbb{R}^2$,
with the fold condition being vacuously satisfied.

For $\varepsilon\neq0$, the spatial differential is
$$
D_{(x,y)}\mathsf{H}_{\varepsilon}(x,y)
=
\begin{pmatrix}
1&0\\
-2\varepsilon(1-\varepsilon x)y&(1-\varepsilon x)^2
\end{pmatrix}.
$$
Consequently,
$$
J_{\varepsilon}(x,y)
:=
\det D_{(x,y)}\mathsf{H}_{\varepsilon}(x,y)
=
(1-\varepsilon x)^2,
$$
and hence
$$
\Sigma_{\varepsilon}
=
\left\{
(x,y)\in\mathbb{R}^2:
x=\frac{1}{\varepsilon}
\right\}.
$$
This is a vertical line. At every point
$$
p
=
\left(\frac{1}{\varepsilon},y\right)
\in\Sigma_{\varepsilon},
$$
we have
$$
D_{(x,y)}\mathsf{H}_{\varepsilon}(p)
=
\begin{pmatrix}
1&0\\
0&0
\end{pmatrix},
$$
and therefore
$$
\ker D_{(x,y)}\mathsf{H}_{\varepsilon}(p)
=
\operatorname{span}\{(0,1)\}.
$$
On the other hand,
$$
T_p\Sigma_{\varepsilon}
=
\operatorname{span}\{(0,1)\}.
$$
Hence, the fold non-degeneracy condition fails dramatically, since
$$
\ker D_{(x,y)}\mathsf{H}_{\varepsilon}(p)
=
T_p\Sigma_{\varepsilon} \neq \{0\}.
$$
Equivalently,
$$
\rank\left(
D_{(x,y)}\mathsf{H}_{\varepsilon}(p)
\vert_{T_p\Sigma_{\varepsilon}}
\right)
=
0,
$$
whereas the tangential-immersion condition requires this rank to equal
$1$. Hence, while each family $\mathsf{\Phi}_{\varepsilon}$ satisfies the
critical-hypersurface condition, it necessarily fails the tangential-immersion condition at every point of its critical set whenever $\varepsilon \neq 0$.

The fold condition fails along the same critical line. Indeed,
$$
D_{(x,y)}J_{\varepsilon}(x,y)
=
\big(
-2\varepsilon(1-\varepsilon x),0
\big),
$$
and therefore
$$
D_{(x,y)}J_{\varepsilon}(p)=0, \qquad \forall p \in \Sigma_{\varepsilon}.
$$
In particular, if
$$
v\in
\ker D_{(x,y)}\mathsf{H}_{\varepsilon}(p) \setminus \{0\},
$$
then
$$
D_{(x,y)}J_{\varepsilon}(p)[v]=0.
$$
Thus $\mathsf{H}_{\varepsilon}$ fails to be fold non-degenerate at
every point of $\Sigma_{\varepsilon}$.

This geometric failure produces, for each nonzero index $\varepsilon \in (-1,1)$, an explicit compact rectifiable set $E_{\varepsilon} \subset \mathbb{R}^2$ for which neither member of $\mathsf{\Phi}_{\varepsilon}$ is a good witness. For each $\varepsilon\neq0$, define
$$
E_{\varepsilon}
:=
\left\{
\left(\frac{1}{\varepsilon},t\right):
0\leq t\leq1
\right\}
\subseteq
\Sigma_{\varepsilon}.
$$
The set $E_{\varepsilon}$ is a compact $1$-rectifiable line segment satisfying
$$
\mathcal{H}^1(E_{\varepsilon})=1.
$$
Nevertheless,
$$
\varphi_{1,\varepsilon}(E_{\varepsilon})
=
\left\{\frac{1}{\varepsilon}\right\}
$$
and
$$
\varphi_{2,\varepsilon}(E_{\varepsilon})
=
\{0\}.
$$
Consequently,
$$
\mathcal{H}^1
\big(
\varphi_{1,\varepsilon}(E_{\varepsilon})
\big)
=
\mathcal{H}^1
\big(
\varphi_{2,\varepsilon}(E_{\varepsilon})
\big)
=
0.
$$
Thus neither member of $\mathsf{\Phi}_{\varepsilon}$ is a good witness for $E_{\varepsilon}$. This conclusion holds for every $\varepsilon\in(0,1)$, which is itself a parameter set of positive $\mathcal{H}^1$ measure.

This example does not contradict Theorem~\ref{thm:localfoldstability-genexample}: the critical line
$
x=\frac{1}{\varepsilon}
$
escapes to infinity as $\varepsilon \rightarrow 0^+$. Hence, on every fixed compact set $K\subseteq\mathbb{R}^2$, the map $\mathsf{H}_{\varepsilon}$ has no critical points for all sufficiently small $\varepsilon$, and
$
\mathsf{H}_{\varepsilon}
\longrightarrow
\mathsf{H}_0
$
in $C^2(K)$. The example instead shows that global fold non-degeneracy of a single base map does not provide a common parameter neighbourhood on which fold non-degeneracy holds throughout all of $\mathbb{R}^2$. Such a global conclusion thus requires a genuinely uniform global hypothesis.
\end{example}

\subsection{Comparison with Peres--Schlag transversality}\label{subsec:peresschlagcomparison-genexample}

There is a useful comparison between Theorem~\ref{thm:localfoldstability-genexample} and the transversality framework for generalized projections introduced by Peres and Schlag \cite{PS2000}. We first recall their framework and then compare its conclusions with those of the present paper.

Let $K\subseteq\mathbb{R}^d$ be compact and consider a parameterized family of scalar mappings
$$
\Pi_\alpha:K\rightarrow\mathbb{R},
\qquad
\alpha\in\mathcal{A}\subseteq\mathbb{R}^m.
$$
For $x\neq y$, Peres and Schlag examine the normalized two-point difference
$$
\Phi_\alpha(x,y)
:=
\frac{\Pi_\alpha(x)-\Pi_\alpha(y)}{|x-y|}
$$
and, at least in the scalar-valued, order-zero case, impose the following transversality condition
\begin{equation}\label{eq:peresschlagorderzero-genexample}
\big|\Phi_\alpha(x,y)\big|\leq c
\qquad\Longrightarrow\qquad
\big|\nabla_\alpha\Phi_\alpha(x,y)\big|\geq c
\end{equation}
for some uniform constant $c>0$. Thus, when the images of two points are close relative to $|x-y|$, variation in $\alpha$ separates them at a controlled first-order rate.

More generally, when $\Pi_\alpha$ takes values in $\mathbb{R}^k$, the parameter derivative is required to have full rank $k$. Quantitatively, Peres and Schlag thus require a lower bound of the form
$$
\det\Big(
D_\alpha\Phi_\alpha(x,y)
D_\alpha\Phi_\alpha(x,y)^{\mathsf{T}}
\Big)
\geq c^2
$$
whenever $\Phi_\alpha(x,y)$ is proportionally small, together with upper bounds for higher parameter derivatives of $\Phi_\alpha$. Their transversality conditions of positive order allow corresponding powers of the scale $|x-y|$ in these inequalities; see \cite[Definitions 2.7 and 7.2]{PS2000}.

The fold condition in Definition~\ref{def:foldnondegeneracy-genexample} is instead a spatial transversality condition. For a fixed tuple $\boldsymbol{\alpha}\in\mathcal{A}^d$, it considers
$
J_{\boldsymbol{\alpha}}(x) = \det D_x\mathsf{H}_{\boldsymbol{\alpha}}(x)
$
and requires
$$
D_xJ_{\boldsymbol{\alpha}}(x)[v]\neq0,
\qquad
\forall v\in\ker D_x\mathsf{H}_{\boldsymbol{\alpha}}(x) \setminus \{0\}.
$$
Consequently, the two frameworks control different degeneracy sets. Peres--Schlag transversality imposes parameter non-degeneracy near
$$
\bigg\{(x,y) \in \mathbb{R}^d \times \mathbb{R}^d :
\left|
\frac{\Pi_{\alpha} (x) - \Pi_{\alpha}(y)}
{\lvert x - y \rvert}
\right|
\leq c
\bigg\}
$$
which is a scale-invariant neighbourhood of the exact incidence set $\Pi_\alpha(x)=\Pi_\alpha(y)$. Fold non-degeneracy instead controls the joint spatial critical set
$$
\big\{
(\boldsymbol{\alpha},x) \in \mathcal{A}^d \times \mathbb{R}^d:
\det D_x\mathsf{H}_{\boldsymbol{\alpha}}(x)=0
\big\},
$$
by requiring the determinant to vary in every nonzero direction in $\ker D_x\mathsf{H}_{\boldsymbol{\alpha}}$.

The two conditions are independent without further hypotheses linking the parameter and spatial variables. Indeed, the prescribed-fold family of Proposition~\ref{prop:prescribedfoldhypersurface-esoteric} is fold non-degenerate for every affinely independent tuple, but the reflected pair
$$
(x',g(x')+s),
\qquad
(x',g(x')-s),
\qquad s\neq0,
$$
has the same image under $\varphi_\alpha^g$ for every $\alpha$. Thus its normalized two-point difference and all of its parameter derivatives vanish, so fold non-degeneracy does not imply Peres--Schlag transversality of any finite order. Conversely, the squared-distance family
$$
\Pi_p(x):=|x-p|^2
$$
satisfies the order-zero condition uniformly, since
$$
\left|\nabla_p \bigg(
\frac{\Pi_p(x)-\Pi_p(y)}{|x-y|}
\bigg)\right|=2
\qquad(x\neq y).
$$
Nevertheless, in $\mathbb{R}^3$ an encoding formed from three collinear pins has identically vanishing determinant and fails the fold condition at its rank-two points. Hence, Peres--Schlag transversality does not by itself imply fold non-degeneracy of an arbitrary finite encoding. The two frameworks are distinct geometric mechanisms with complementary projection-theoretic consequences.

Given a finite measure $\mu$ supported on some compact fractal set $K \subseteq \mathbb{R}^d$, Peres and Schlag study the pushforward measures
$$
\nu_\alpha
:=
(\Pi_\alpha)_\#\mu.
$$
The main technique in \cite{PS2000} controls averaged Sobolev norms of $\nu_\alpha$ over $\alpha\in\mathcal{A}$. This gives dimensional bounds for the exceptional parameters at which the Sobolev dimension of $\nu_\alpha$ drops. Suppose that $\mathcal{A}\subseteq\mathbb{R}^m$ is open and that $\mu$ has finite $\alpha_0$-energy for some $\alpha_0>1$. In the scalar-valued, order-zero case, their higher-dimensional parameter theorem gives, under the relevant transversality and regularity hypotheses, the model estimate
\begin{equation}\label{eq:peresschlagexceptional-genexample}
\dim\big\{
\alpha\in\mathcal{A}:
\dim_{\mathrm{s}}(\nu_\alpha)\leq1
\big\}
\leq
m+1-\alpha_0;
\end{equation}
see \cite[Theorem 7.3]{PS2000}. Since a measure on $\mathbb{R}$ having Sobolev dimension greater than one is absolutely continuous \cite[Definition 2.3]{PS2000}, the same expression bounds the dimension of parameters for which the image has zero Lebesgue measure. For a compact set of Hausdorff dimension $s>1$, Frostman's lemma permits a choice of $\mu$ with finite $\alpha_0$-energy for every fixed $1<\alpha_0<s$; passing to the limit $\alpha_0\uparrow s$ yields the familiar bound $m+1-s$. At the rectifiable endpoint $s=1$, however, the limiting expression is the ambient bound $m=\dim\mathcal{A}$ and gives no further structural information about the exceptional set.

Theorem~\ref{thm:dcurvesrefined-genexample} works precisely at this endpoint, with its stronger conclusion arising from the additional geometric information supplied by rectifiability and by the finite encoding map: every admissible tuple must contain a good witness. Such a conclusion should not be expected without rectifiability, since the method relies on the local Lipschitz parametrizations and approximate tangent structure of rectifiable sets. Thus the Peres--Schlag method supplies quantitative exceptional-set estimates above the rectifiable endpoint $s=1$, whereas the present method supplies deterministic and, in several examples, structural information at the endpoint itself.

Both methods turn quantitative transversality into openness and stability. In the Peres--Schlag framework, the lower bound is placed on the parameter derivative of a normalized two-point difference. In Theorem~\ref{thm:localfoldstability-genexample}, it is placed on the spatial derivative of the Jacobian determinant in its kernel direction. When both calculations are available, they provide genuinely different information: Theorem~\ref{thm:localfoldstability-genexample} gives stable finite collections of witness maps, while Peres--Schlag theory can additionally bound the Hausdorff dimension of the full exceptional parameter set.

\begin{remark}
The author will treat the general case of $C^2$ images of rectifiable $k$-sets, for $k\in\{2,\ldots,d-1\}$, in forthcoming work. For now, however, we continue to refer to our results as concerning generalized \textit{curve} projections, as opposed to \textit{surface} projections, at the rectifiable endpoint.
\end{remark}

\section{Generalized distances, smoothed polyhedral norms, and \texorpdfstring{$C^2$}{C2} graphs}
\label{sec:exoticcurveprojections-esoteric}

This section concludes the paper with a few more exotic examples than pinned distance functions and radial projections. The first concerns nonlinear
anisotropic pinned-distance functionals. The second is a smooth model
for pinned distances generated by a polyhedral norm. The third shows
that an admissible family may have a critical hypersurface with
essentially arbitrary prescribed $C^2$ geometry. In each case, the
Jacobian determinant and fold non-degeneracy are computed explicitly, thus providing the reader with a solid library of examples for understanding the more abstract techniques introduced in Section \ref{sec:generalizedcurveprojections-genexample}.

\subsection{Nonlinear anisotropic distance functionals}

Distances generated by convex bodies have been studied extensively in
connection with Falconer's distance problem; see, for example,
\cite{IL2005}, which considers these variants of Falconer's distance
problem, as well as other discrete variants. The first class considered
here consists of nonlinear anisotropic distance functionals for which
the fold calculation remains completely explicit. The necessary
notation is fixed next.

Let $A$ be an invertible $d\times d$ matrix, and let
$\eta: [0,\infty)\rightarrow \mathbb{R}$ be any $C^2$ function whose derivative $\eta'$ is everywhere non-vanishing. For each $q\in\mathbb{R}^d$, define
$$
\varphi_q^{A,\eta}(x)
:=
\eta\big(|A(x-q)|^2\big),
\qquad
x\in\mathbb{R}^d,
$$
where, as before, $\lvert \cdot \rvert$ denotes the standard Euclidean distance on $\mathbb{R}^d$. These mappings include squared ellipsoidal distances and their
corresponding nonlinear reparametrizations by $\eta$.

\begin{proposition}[Nonlinear Anisotropic Distance Functionals]
\label{prop:generalizeddistancestability-esoteric}
Let
$
\mathcal{Q} := \{q^{(1)},\ldots,q^{(d)}\} \subseteq\mathbb{R}^d
$
be an affinely independent point set. Then the family
$$
\mathsf{\Phi}_{\mathcal{Q}}^{A,\eta}
:=
\big\{
\varphi_{q^{(1)}}^{A,\eta},
\ldots,
\varphi_{q^{(d)}}^{A,\eta}
\big\}
$$
is admissible on $\mathbb{R}^d$. More precisely, the associated critical set is
$
\Sigma_{\mathcal{Q}}^{A,\eta} = \Aff(\mathcal{Q}),
$
and, at each of these critical points, the corresponding encoding map is fold non-degenerate.
\end{proposition}

\begin{proof}
For each $j\in\{1,\ldots,d\}$, set
$$
s_j(x):=|A(x-q^{(j)})|^2
$$
and define the squared anisotropic encoding map by
$$
\mathsf{H}_{\mathcal{Q}}^{A,2}(x)
:=
\big(
s_1(x),\ldots,s_d(x)
\big).
$$
The chain rule gives
$$
D_x\varphi_{q^{(j)}}^{A,\eta}(x)
=
\eta'\big(s_j(x)\big)D_xs_j(x).
$$
Consequently,
\begin{equation}\label{eq:anisotropicdiagonalfactorization-esoteric}
D_x\mathsf{H}_{\mathcal{Q}}^{A,\eta}(x)
=
\mathsf{D}_{\eta}(x)
D_x\mathsf{H}_{\mathcal{Q}}^{A,2}(x),
\end{equation}
where $\mathsf{D}_{\eta}(x)$ is the $d \times d$ diagonal matrix with entries
$
\eta'(s_j(x))
$
for each
$
j=1,\ldots,d.
$
Since $\eta'$ is non-vanishing, we know that
$\mathsf{D}_{\eta}(x)$ is invertible for every $x$.

We first analyze the squared encoding map. Make the invertible change of variables
$
y:=Ax
$
and set
$
p^{(j)}:=Aq^{(j)},
$
and observe that the point set $\mathcal{P} := \{p^{(1)},\ldots,p^{(d)}\}$ is affinely independent. In the
$y$ variable, the squared encoding map thus becomes
$$
\widetilde{\mathsf{H}}_{\mathcal{Q}}(y)
:=
\big(
|y-p^{(1)}|^2,\ldots,|y-p^{(d)}|^2
\big).
$$
After an appropriate elementary row operation, its Jacobian determinant thus satisfies
\begin{align}
\det D_y\widetilde{\mathsf{H}}_{\mathcal{Q}}(y)
&=
2^d
\det
\begin{pmatrix}
(y-p^{(1)})^{\mathsf{T}}\\
(p^{(1)}-p^{(2)})^{\mathsf{T}}\\
\vdots\\
(p^{(1)}-p^{(d)})^{\mathsf{T}}
\end{pmatrix}
\notag\\
&=
c_{\mathcal{Q}}
\big\langle
y-p^{(1)},n_{\mathcal{Q}}
\big\rangle,
\label{eq:anisotropicdistancedeterminant-esoteric}
\end{align}
where $n_{\mathcal{Q}}$ is a unit normal vector to
$\Aff (\mathcal{P})$ and we are guaranteed that
$c_{\mathcal{Q}}\neq0$ since the point set $\mathcal{P}$ is affinely independent.

The preceding calculation shows that the critical set of
$\widetilde{\mathsf{H}}_{\mathcal{Q}}$ is precisely
$
\Aff \big( \mathcal{P}\big).
$
At a point of this hyperplane, the row space of
$D_y\widetilde{\mathsf{H}}_{\mathcal{Q}}$ is precisely its tangent space; therefore,
$$
\ker D_y\widetilde{\mathsf{H}}_{\mathcal{Q}}(y)
=
\operatorname{span}\{n_{\mathcal{Q}}\}.
$$
Moreover, the identity \eqref{eq:anisotropicdistancedeterminant-esoteric} gives
$$
D_y\big(
\det D_y\widetilde{\mathsf{H}}_{\mathcal{Q}}
\big)(y)
[n_{\mathcal{Q}}]
=
c_{\mathcal{Q}}
\neq0.
$$
Thus, the canonical encoding map is fold non-degenerate in the $y$ variable.
The displayed kernel and determinant calculations are precisely the
kernel-transversality criterion, as verified in Proposition~\ref{prop:kerneltransversality-diffgeo}.

Returning to the $x$ variable, we have
$$
D_x\mathsf{H}_{\mathcal{Q}}^{A,2}(x)
=
D_y\widetilde{\mathsf{H}}_{\mathcal{Q}}(Ax)A.
$$
Hence its critical set is
$$
A^{-1}
\Aff \big(\mathcal{P}\big)
=
\Aff(\mathcal{Q}).
$$
If $x$ belongs to this critical hyperplane, then
$$
\ker D_x\mathsf{H}_{\mathcal{Q}}^{A,2}(x)
=
\operatorname{span}
\{A^{-1}n_{\mathcal{Q}}\}.
$$
Writing
$$
J_{A,2}(x)
:=
\det D_x\mathsf{H}_{\mathcal{Q}}^{A,2}(x),
$$
the chain rule and \eqref{eq:anisotropicdistancedeterminant-esoteric} give
$$
D_xJ_{A,2}(x)[A^{-1}n_{\mathcal{Q}}]
=
\det(A)c_{\mathcal{Q}}
\neq0.
$$
Thus the squared anisotropic encoding is fold non-degenerate.

Finally, let
$$
J_{A,\eta}(x)
:=
\det D_x\mathsf{H}_{\mathcal{Q}}^{A,\eta}(x).
$$
By \eqref{eq:anisotropicdiagonalfactorization-esoteric},
$$
J_{A,\eta}(x)
=
\det\mathsf{D}_{\eta}(x)J_{A,2}(x).
$$
The invertibility of $\mathsf{D}_{\eta}(x)$ shows that the two encodings have the same critical set and the same kernel at every
critical point. Hence, along this common critical set, $J_{A,2}=0$, and, for
every kernel vector $v$,
$$
D_xJ_{A,\eta}(x)[v]
=
\det\mathsf{D}_{\eta}(x)
D_xJ_{A,2}(x)[v] \neq 0.
$$
The fold non-degeneracy follows. Lemma~\ref{lem:foldimpliestangentialimmersion-genexample} then gives the
critical-hypersurface and tangential-immersion conditions.
Equivalently, the final kernel calculation can be inferred by referencing
Proposition~\ref{prop:kerneltransversality-diffgeo}.\qedhere
\end{proof}

The following is an immediate corollary of Section~\ref{sec:generalizedcurveprojections-genexample}. It also provides
conceptual justification for the more abstract framework developed
there, since it shows that the preceding pinned-distance results are
stable under every $C^2$ reparametrization $\eta$ whose derivative is
nowhere zero.

\begin{corollary}[Exceptional Parameters for Anisotropic Distances]
\label{cor:generalizeddistanceexceptionalset-esoteric}
Let $E\subseteq\mathbb{R}^d$ be a Borel $1$-rectifiable set satisfying
$\mathcal{H}^1(E)>0$, and define
$$
\mathcal{B}_E^{A,\eta}
:=
\left\{
q\in\mathbb{R}^d:
\mathcal{H}^1
\big(
\{\eta(|A(x-q)|^2):x\in E\}
\big)
=0
\right\}.
$$
Then $\mathcal{B}_E^{A,\eta}$ is contained in an affine subspace of
dimension at most $d-2$.
\end{corollary}

\begin{proof}
Proposition~\ref{prop:generalizeddistancestability-esoteric} shows that every
affinely independent $d$-tuple of parameters is admissible. The result
therefore follows immediately from Proposition~\ref{prop:parameterizedexceptionalset-genexample}.\qedhere
\end{proof}

\subsection{Bregman functionals associated with polyhedral norms}

We next construct a smooth generalized pinned-distance model associated
to an arbitrary polyhedral norm. Rather than considering translates of a
regularized norm, we consider the Bregman functionals generated by a
smooth, strictly convex regularization. We refer to
\cite{KY2008,Schneider2014} for standard background on convex bodies, polar duality, support functions, and Minkowski functionals.  Bregman functionals originate in \cite{Bregman1967}; a modern treatment from
the viewpoint of convex analysis and Legendre functions is given in \cite{BB1997}.

This example is particularly useful because it combines a transparent
geometric model with an equally explicit analytic non-degeneracy
mechanism: the level sets of a Bregman functional
generated by a strictly convex potential are smooth strictly convex
hypersurfaces, and the relevant pin parameters become affine after
passage to the dual coordinates $\nabla F(\alpha)$. Analytically, the
positive-definite Hessian $D^2F$ identifies the unique missing direction
of a critical encoding and forces the Jacobian determinant to cross its
critical hypersurface transversely. Consequently, in this polyhedral
setting, the log-sum-exp regularization preserves the finite collection
of supporting directions while supplying a globally smooth,
positive-definite Hessian. Thus the example retains the geometric
content of a smoothed polyhedral norm while also possessing the
analytic properties required by the fold non-degeneracy and
finite-witness framework developed here.

\subsubsection{The argument for strictly convex functions}
We first record the differential-geometric calculation in a form which applies to any strictly convex potential. We then specialize to smooth approximations of norms $\rho_P$ associated with convex polytopes $P \subset \mathbb{R}^d$.

The notation is fixed before a general version of the result is
stated. Let $F:\mathbb{R}^d\rightarrow\mathbb{R}$ be a $C^2$ function
whose Hessian matrix
$
D^2F(x)
$
is positive definite for every $x\in\mathbb{R}^d$. In particular, $F$ is strictly convex. For each $\alpha\in\mathbb{R}^d$, define the \textbf{pinned Bregman functional} by writing
$$
\varphi_\alpha^F(x)
:=
F(x)-F(\alpha)
-
\big\langle
\nabla F(\alpha),x-\alpha
\big\rangle.
$$
The strict convexity of $F$ then implies that
$
\varphi_\alpha^F(x) > 0
$
for all $x \in \mathbb{R}^d \setminus \{\alpha\}$, and $\varphi_{\alpha}^F (\alpha) = 0$. 

We then have the following fold non-degeneracy result for the pinned Bregman functional associated with $F$.

\begin{proposition}[Bregman Folds from Strictly Convex Potentials]
\label{prop:bregmanfold-esoteric}
Let
$$
\boldsymbol{\alpha}
:=
(\alpha_1,\ldots,\alpha_d)
\in(\mathbb{R}^d)^d,
$$
and set
$$
\xi_j:=\nabla F(\alpha_j),
\qquad
j=1,\ldots,d.
$$
Suppose that $\xi_1,\ldots,\xi_d$ are affinely independent. Then the family
$$
\mathsf{\Phi}_{\boldsymbol{\alpha}}^F
:=
\big\{
\varphi_{\alpha_1}^F,\ldots,\varphi_{\alpha_d}^F
\big\}
$$
is admissible on $\mathbb{R}^d$.  More precisely, the critical set of
its encoding map is
$$
\Sigma_{\boldsymbol{\alpha}}^F
=
(\nabla F)^{-1}
\big(
\Aff\{\xi_1,\ldots,\xi_d\}
\big)
$$
and, at every point $x \in \Sigma_{\boldsymbol{\alpha}}^F$, the family $\mathsf{\Phi}_{\boldsymbol{\alpha}}^F$ is fold non-degenerate.
\end{proposition}

\begin{proof}
Let
$
\mathsf{H}_{\boldsymbol{\alpha}}^F := \big( \varphi_{\alpha_1}^F,\ldots,\varphi_{\alpha_d}^F \big).
$
For each $j\in\{1,\ldots,d\}$, differentiation in the spatial variable gives
$$
D_x\varphi_{\alpha_j}^F(x)
=
\big(
\nabla F(x)-\xi_j
\big)^{\mathsf{T}}.
$$
We then set
$$
L_{\boldsymbol{\alpha}}
:=
\operatorname{span}
\{\xi_1-\xi_2,\ldots,\xi_1-\xi_d\}.
$$
and recall that
$
\dim L_{\boldsymbol{\alpha}}=d-1,
$
by the assumption of affine independence.
Choose any unit vector
$
n_{\boldsymbol{\alpha}}\in L_{\boldsymbol{\alpha}}^\perp
$
and define $c_{\boldsymbol{\alpha}} \in \mathbb{R}$ by setting
$$
c_{\boldsymbol{\alpha}}
:=
\det
\begin{pmatrix}
n_{\boldsymbol{\alpha}}^{\mathsf{T}}\\
(\xi_1-\xi_2)^{\mathsf{T}}\\
\vdots\\
(\xi_1-\xi_d)^{\mathsf{T}}
\end{pmatrix}.
$$
The vectors appearing as rows in this matrix form a basis of
$\mathbb{R}^d$, and hence
$
c_{\boldsymbol{\alpha}}\neq0.
$

We then use the notation above in the following calculation. Again, by elementary row operations, we observe that
\begin{align*}
J_{\boldsymbol{\alpha}}^F(x) :=
\det D_x\mathsf{H}_{\boldsymbol{\alpha}}^F(x) =
\det
\begin{pmatrix}
\big(\nabla F(x)-\xi_1\big)^{\mathsf{T}}\\
(\xi_1-\xi_2)^{\mathsf{T}}\\
\vdots\\
(\xi_1-\xi_d)^{\mathsf{T}}
\end{pmatrix}
\end{align*}
and so
\begin{equation}\label{eq:bregmandeterminant-esoteric}
J_{\boldsymbol{\alpha}}^F(x) =c_{\boldsymbol{\alpha}}
\big\langle
\nabla F(x)-\xi_1,n_{\boldsymbol{\alpha}}
\big\rangle
\end{equation}

The preceding calculation shows that
$
J_{\boldsymbol{\alpha}}^F(x)=0
$
if and only if
$$
\nabla F(x)
\in
\big( \xi_1+L_{\boldsymbol{\alpha}} \big)
:=
\Aff\{\xi_1,\ldots,\xi_d\}.
$$
This proves the asserted formula for
$
\Sigma_{\boldsymbol{\alpha}}^F.
$

We now proceed to prove the associated fold non-degeneracy at all critical points. Define
$$
g_{\boldsymbol{\alpha}}(x)
:=
\big\langle
\nabla F(x)-\xi_1,n_{\boldsymbol{\alpha}}
\big\rangle.
$$
Then, for every $v\in\mathbb{R}^d$, a direct calculation gives
$$
D_xg_{\boldsymbol{\alpha}}(x)[v]
=
\big\langle
D^2F(x)v,n_{\boldsymbol{\alpha}}
\big\rangle,
$$
and, in particular,
\begin{equation}\label{eq:bregmannormalderivative-esoteric}
D_xg_{\boldsymbol{\alpha}}(x)[n_{\boldsymbol{\alpha}}]
=
\big\langle
D^2F(x)n_{\boldsymbol{\alpha}},n_{\boldsymbol{\alpha}}
\big\rangle
>0.
\end{equation}
Consequently, zero is a regular value of $g_{\boldsymbol{\alpha}}$, which implies that the set of critical points
$
\Sigma_{\boldsymbol{\alpha}}^F = g_{\boldsymbol{\alpha}}^{-1}(0)
$
is a $C^1$ embedded hypersurface; see Subsection
\ref{subsec:hypersurfacetransversality-diffgeo}.

Hence, if we take
$x\in\Sigma_{\boldsymbol{\alpha}}^F$, then it may be assumed, after
elementary row operations, that the first row of
$
D_x\mathsf{H}_{\boldsymbol{\alpha}}^F(x)
$
belongs to $L_{\boldsymbol{\alpha}}$, while the remaining row differences
span $L_{\boldsymbol{\alpha}}$.  Therefore,
$$
\rank D_x\mathsf{H}_{\boldsymbol{\alpha}}^F(x)
=
d-1
$$
and
\begin{equation}
\ker D_x\mathsf{H}_{\boldsymbol{\alpha}}^F(x)
=
\operatorname{span}\{n_{\boldsymbol{\alpha}}\}.
\label{eq:bregmankernel-esoteric}
\end{equation}
Differentiating \eqref{eq:bregmandeterminant-esoteric} in this kernel
direction gives
$$
D_xJ_{\boldsymbol{\alpha}}^F(x)[n_{\boldsymbol{\alpha}}]
=
c_{\boldsymbol{\alpha}}
\big\langle
D^2F(x)n_{\boldsymbol{\alpha}},n_{\boldsymbol{\alpha}}
\big\rangle
\neq0,
$$
which establishes fold non-degeneracy at every
$x\in\Sigma_{\boldsymbol{\alpha}}^F$.

Even though we know that fold non-degeneracy implies the tangential-immersion condition, we still verify the tangential-immersion condition directly. By the regular-level-set identity
\eqref{eq:regularleveltangent-diffgeo}, the tangent space of the
critical hypersurface is
$$
T_x\Sigma_{\boldsymbol{\alpha}}^F
=
\left\{
v\in\mathbb{R}^d:
\big\langle
D^2F(x)v,n_{\boldsymbol{\alpha}}
\big\rangle
=0
\right\}.
$$
Equation \eqref{eq:bregmannormalderivative-esoteric} shows that
$
n_{\boldsymbol{\alpha}}\notin T_x\Sigma_{\boldsymbol{\alpha}}^F.
$
Together with \eqref{eq:bregmankernel-esoteric}, this proves the
tangential-immersion condition directly, by Proposition~\ref{prop:kerneltransversality-diffgeo}.\qedhere
\end{proof}

\subsubsection{Specialization to $C^2$ Bregman approximations of polyhedral norms}
We now apply Proposition~\ref{prop:bregmanfold-esoteric} to obtain good-witness conditions for a $C^2$ Bregman approximation to an arbitrary polyhedral norm. Let $\rho : \mathbb{R}^d \rightarrow [0,\infty)$ be any polyhedral norm on $\mathbb{R}^d$, and set
$$
P:=\{z\in\mathbb{R}^d:\rho(z)\leq1\}.
$$
Since $\rho$ is a \emph{polyhedral} norm, this set $P \subset \mathbb{R}^d$ is compact, convex, centrally symmetric, and also contains the origin in its interior. We then define the \textbf{polar polytope}
$$
P^{\circ} := \big\{\xi \in \mathbb{R}^d : \langle \xi , x\rangle \leq 1 \text{ for every } x \in P \big\}.
$$
One should think of $P^{\circ}$ as identifying all possible candidate normal vectors $\xi$ to hyperplanes supporting the polyhedron $P$. To further emphasize this geometric duality, if $\xi \in \partial P^{\circ}$, then the associated hyperplane
$$
H_{\xi} := \big\{x \in \mathbb{R}^d : \langle \xi, x \rangle = 1 \big\}
$$
satisfies
$$
H_{\xi} \cap \partial P \neq \emptyset, \qquad H_{\xi} \cap \text{int} \big(P \big) = \emptyset,
$$
so that $H_{\xi}$ intersects the boundary of $P$, without intersecting its interior.

Now, if $u_1,\ldots,u_N$ are the vertices of the polar polytope $P^\circ$, then the associated polyhedral norm $\rho$ satisfies
\begin{equation}\label{eq:polynormvertexidentity-esoteric}
\rho(z)
=
\max_{1\leq r\leq N}
\langle u_r,z\rangle, \qquad \forall z \in \mathbb{R}^d.
\end{equation}
The identity \eqref{eq:polynormvertexidentity-esoteric} is the standard identification of the Minkowski functional of $P$ with the support function of its polar body; see, for example,
\cite{KY2008,Schneider2014}.
For each $\varepsilon>0$, define the log-sum-exp regularization
\begin{equation}
F_\varepsilon(z)
:=
\varepsilon
\log
\left(
\frac{1}{N}
\sum_{r=1}^N
\exp
\left(
\frac{\langle u_r,z\rangle}{\varepsilon}
\right)
\right).
\label{eq:polyhedrallogsumexp-esoteric}
\end{equation}
A simple elementary logarithmic-sum-exponential bound  together with the identity \eqref{eq:polynormvertexidentity-esoteric} give the asymptotic estimate
\begin{equation}
\rho(z)-\varepsilon\log N
\leq
F_\varepsilon(z)
\leq
\rho(z).
\label{eq:polyhedrallogsumexpbound-esoteric}
\end{equation}
Consequently, the function $F_\varepsilon$ converges uniformly to $\rho$ as
$\varepsilon\downarrow0$.  For the log-sum-exp function as a standard
smooth convex approximation to a finite maximum, see
\cite{BV2004}. 

Thus $F_\varepsilon$ gives a globally smooth and strictly convex approximation to the polyhedral norm $\rho$. The associated Bregman functionals provide the generalized pinned family considered below. The associated geometry induced by their level surfaces is explored in detail in Proposition \ref{prop:smoothvariationbregmanspheres-esoteric} and its corresponding proof. For now, we digress, and proceed with our calculations.

For each $z\in\mathbb{R}^d$, set
$$
p_r(z)
:=
\frac{
\exp\big(\langle u_r,z\rangle/\varepsilon\big)
}{
\displaystyle
\sum_{s=1}^N
\exp\big(\langle u_s,z\rangle/\varepsilon\big)
},
\qquad
\overline{u}(z)
:=
\sum_{r=1}^N p_r(z)u_r.
$$
A direct calculation then gives that
$$
\nabla F_\varepsilon(z)
=
\overline{u}(z)
$$
and
\begin{equation}\label{eq:polyhedrallogsumexphessian-esoteric}
D^2F_\varepsilon(z)
=
\frac{1}{\varepsilon}
\sum_{r=1}^N
p_r(z)
\big(
u_r-\overline{u}(z)
\big)
\otimes
\big(
u_r-\overline{u}(z)
\big),
\end{equation}
where $u\otimes v:=uv^{\mathsf{T}}$ denotes the outer-product matrix associated with $u,v \in \mathbb{R}^d$.

Since the vertices of $P^{\circ}$ affinely span $\mathbb{R}^d$, every
nonzero $v\in\mathbb{R}^d$ satisfies
\begin{align*}
\big\langle
D^2F_\varepsilon(z)v,v
\big\rangle
=
\frac{1}{\varepsilon}
\sum_{r=1}^N
p_r(z)
\big\langle
u_r-\overline{u}(z),v
\big\rangle^2
>0.
\end{align*}
Thus $D^2F_\varepsilon(z)$ is positive definite for every $z\in\mathbb{R}^d$, and $F_{\varepsilon}$ is strictly convex.

The gradient mapping $\nabla F_\varepsilon$ is injective. This follows from the Fundamental Theorem of Calculus and the previous lower bound, since whenever $x \neq y$ one obtains
\begin{align*}
\big\langle
\nabla F_\varepsilon(x)-\nabla F_\varepsilon(y),
x-y
\big\rangle
=
\int_0^1
\big\langle
D^2F_\varepsilon\big(y+t(x-y)\big)(x-y),
x-y
\big\rangle
\,dt>0.
\end{align*}
Hence, injectivity follows from strict monotonicity of the gradient of a differentiable strictly convex function; see \cite{Rockafellar1970} for a standard reference. As its derivative is invertible, the function
$
\nabla F_\varepsilon
$
is a $C^\infty$ diffeomorphism from $\mathbb{R}^d$ onto
the open set
$$
\mathcal{A}_\varepsilon
:=
\nabla F_\varepsilon(\mathbb{R}^d)
\subseteq\operatorname{int}(P^\circ).
$$
In particular, $\mathcal{A}_\varepsilon$ contains affinely
independent $d$-tuples, and each such $d$-tuple has a unique preimage under $\nabla F_\varepsilon$.

For each $\alpha\in\mathbb{R}^d$, define
\begin{equation}
\varphi_\alpha^\varepsilon(x)
:=
F_\varepsilon(x)-F_\varepsilon(\alpha)
-
\big\langle
\nabla F_\varepsilon(\alpha),x-\alpha
\big\rangle.
\label{eq:polyhedralbregmanfunctional-esoteric}
\end{equation}
This is the Bregman functional generated by the regularized
polyhedral norm $F_\varepsilon$; see
\cite{Bregman1967,BB1997}.
Since $P^\circ$ is centrally symmetric, the vertex set
$\{u_1,\ldots,u_N\}$ is centrally symmetric.  It follows that
$F_\varepsilon$ is even,
$
F_\varepsilon(0)=0,
$
and
$
\nabla F_\varepsilon(0)=0.
$
Consequently,
$$
\varphi_0^\varepsilon(x)=F_\varepsilon(x),
$$
so the functional pinned at the origin satisfies the uniform
approximation \eqref{eq:polyhedrallogsumexpbound-esoteric}.

\begin{remark}
Although $\alpha$ plays the role of a pin for the Bregman functional, it should not generally be interpreted as translating the polytope. Rather, $\alpha$ specifies the point at which the convex potential $F_\varepsilon$ is linearized. Indeed, if
\[
\ell_\alpha(x)
:=
F_\varepsilon(\alpha)
+
\left\langle
\nabla F_\varepsilon(\alpha),x-\alpha
\right\rangle,
\]
then
\[
\varphi_\alpha^\varepsilon(x)
=
F_\varepsilon(x)-\ell_\alpha(x).
\]
Thus, the function $\varphi_\alpha^\varepsilon(x)$ measures the vertical separation between the graph of $F_\varepsilon$ at $x$ and the associated tangent hyperplane to the graph of $F_\varepsilon$ at $\alpha \in \mathbb{R}^d$. In particular,
\[
\varphi_\alpha^\varepsilon(\alpha)=0,
\qquad
D_x\varphi_\alpha^\varepsilon(\alpha)=0,
\]
and strict convexity implies that $\alpha$ is the unique point at which the functional vanishes. It is in this sense that $\alpha$ is the pin of the generalized distance functional.

The geometric role of $\alpha \in \mathbb{R}^d$ is even more apparent if we consider the associated polar dual parameter
\[
\xi_\alpha:=\nabla F_\varepsilon(\alpha)
\in\operatorname{int}(P^\circ).
\]
Indeed, all of the weights $p_r(\alpha)$ are strictly positive, and the vertices $u_1,\ldots,u_N$ affinely span $\mathbb{R}^d$, so their weighted average lies in the interior of $P^\circ$.
For the logarithmic-sum-exponential regularization of the polyhedral norm, we then have the associated dual expressions
\[
\xi_\alpha
=
\sum_{r=1}^N
p_r(\alpha)u_r,
\qquad
p_r(\alpha)
=
\frac{
e^{\langle u_r,\alpha\rangle/\varepsilon}
}{
\sum_{s=1}^N
e^{\langle u_s,\alpha\rangle/\varepsilon}
}.
\]
Consequently, $\alpha$ determines a weighted selection of the supporting covectors $u_1,\ldots,u_N$ (which are vertices of the polar polytope $P^{\circ}$) of the polyhedral norm. 

In particular, if $\alpha_0 \in \mathbb{R}^d$ is some parameter which lies deeply inside a cone on which one covector $u_k$ strongly realizes the polyhedral norm, then the associated parameter $\xi_{\alpha_0}$ will be close in distance to $u_k$. Accordingly, varying $\alpha$ changes the supporting direction selected from the polar polytope; it does not simply translate the original polytope.
\end{remark}

The following good-witness result concerns tuples
$\boldsymbol{\alpha} := (\alpha_1,\ldots,\alpha_d)$ whose associated
dual parameters are affinely independent. Its conclusions follow from
the refined $d$-curve result in Section~\ref{sec:generalizedcurveprojections-genexample} and a direct
calculation of the derivatives of the associated functions
$\varphi_{\alpha_j}^{\varepsilon}$.

\begin{corollary}[Bregman Functionals for Arbitrary Polyhedral Norms]
\label{cor:polyhedralbregmanwitness-esoteric}
Let $d\geq2$, let $\rho$ be any polyhedral norm on
$\mathbb{R}^d$, and fix $\varepsilon>0$.  Suppose that
$$
\boldsymbol{\alpha}
=
(\alpha_1,\ldots,\alpha_d)
\in(\mathbb{R}^d)^d
$$
is chosen so that
$$
\nabla F_\varepsilon(\alpha_1),
\ldots,
\nabla F_\varepsilon(\alpha_d)
$$
are affinely independent.  Then
$
\{\varphi_{\alpha_1}^\varepsilon,\ldots, \varphi_{\alpha_d}^\varepsilon\}
$
is admissible on $\mathbb{R}^d$, and every critical point of its
encoding map is fold non-degenerate.  In particular, if
$E\subseteq\mathbb{R}^d$ is a Borel $1$-rectifiable set satisfying
$\mathcal{H}^1(E)>0$, then there exists
$j\in\{1,\ldots,d\}$ such that
$$
\mathcal{H}^1
\big(
\varphi_{\alpha_j}^\varepsilon(E)
\big)
>0.
$$
\end{corollary}

\begin{proof}
The Hessian calculation
\eqref{eq:polyhedrallogsumexphessian-esoteric} verifies the hypotheses of
Proposition~\ref{prop:bregmanfold-esoteric}.  The admissibility and fold
non-degeneracy therefore follow, and the positive-image conclusion follows from Theorem~\ref{thm:dcurvesrefined-genexample}.\qedhere
\end{proof}

Corollary~\ref{cor:polyhedralbregmanwitness-esoteric}
may be interpreted as a pinned-radius theorem for the associated convex Bregman geometry generated by $F_\varepsilon$. Indeed, for each $\alpha\in\mathbb{R}^d$ and $t\geq0$, define the associated Bregman
ball and Bregman sphere by
\[
\mathcal{B}_{\alpha,\varepsilon}(t)
:=
\left\{
x\in\mathbb{R}^d:
\varphi_\alpha^\varepsilon(x)\leq t
\right\}
\]
and
\[
\mathcal{S}_{\alpha,\varepsilon}(t)
:=
\left\{
x\in\mathbb{R}^d:
\varphi_\alpha^\varepsilon(x)=t
\right\} = \partial \mathcal{B}_{\alpha, \varepsilon} (t).
\]
Since $F_\varepsilon$ is smooth and strictly convex, the sets
$\mathcal{B}_{\alpha,\varepsilon}(t)$ form a nested family of convex
bodies satisfying
\[
\mathcal{B}_{\alpha,\varepsilon}(0)=\{\alpha\},
\]
and their boundaries are smooth, strictly convex hypersurfaces whenever $t>0$. However, since $F_{\varepsilon}$ is only an approximation of some polyhedral norm, these need not be centrally symmetric, nor are they generally translates of one fixed body.

For a set $E\subseteq\mathbb{R}^d$, its image under the Bregman
functional is exactly the set of radii $t\geq0$ for which the
corresponding Bregman sphere intersects $E$:
$$
\varphi_\alpha^\varepsilon(E)
:=
\left\{
t\geq0:
E\cap\mathcal{S}_{\alpha,\varepsilon}(t)\neq\varnothing
\right\}.
$$
Consequently, Corollary~\ref{cor:polyhedralbregmanwitness-esoteric} says that, given $d$ pins whose associated dual parameters
$
\xi_j := \nabla F_\varepsilon(\alpha_j)
$
with
$
j=1,\ldots,d,
$
are affinely independent, every positive-length $1$-rectifiable set determines a positive-measure set of Bregman radii from at least one of those pins.

The appearance of the dual parameters $\xi_j$ is geometrically natural. Indeed,
\[
\varphi_{\alpha_j}^\varepsilon(x)
=
F_\varepsilon(x)
-
\langle\xi_j,x\rangle
+
c_j,
\]
where
\[
c_j
=
-F_\varepsilon(\alpha_j)
+
\langle\xi_j,\alpha_j\rangle
\]
is independent of $x$. Thus the essential data associated with the pin
$\alpha_j$ are the slope $\xi_j$ of the tangent hyperplane to the graph
of $F_\varepsilon$ at $\alpha_j$. Affine independence of
$\xi_1,\ldots,\xi_d$ means that these tangent hyperplanes probe the
graph of $F_\varepsilon$ in $d$ independent dual directions.

This interpretation can also be expressed infinitesimally. If
$\tau_E(x)$ is an approximate unit tangent vector to $E$ at $x$, then
\[
D_x\varphi_{\alpha_j}^\varepsilon(x)
\big[
\tau_E(x)
\big]
=
\left\langle
\nabla F_\varepsilon(x)-\xi_j,
\tau_E(x)
\right\rangle.
\]
Therefore, the $j$-th Bregman functional detects tangential variation
in $E$ precisely where $E$ crosses the corresponding Bregman spheres,
rather than running tangentially along them. In view of the
one-dimensional Area Formula, the positive-image conclusion means
that this tangential derivative is nonzero on a positive
$\mathcal{H}^1$-measure portion of $E$ for at least one index $j$. The
critical-hypersurface and tangential-immersion conditions ensure that
a positive-length portion of $E$ cannot remain tangentially invisible
to all $d$ families, even along the critical set of their canonical
encoding map. The fold non-degeneracy established in the corollary
additionally shows that this geometric configuration persists under
small perturbations of the pins.

There is also a useful, if somewhat heuristic, interpretation as
$\varepsilon\rightarrow 0^+$. For a fixed pin $\alpha$, let
\[
I(\alpha)
:=
\left\{
r:
\langle u_r,\alpha\rangle=\rho(\alpha)
\right\}
\]
and set
\[
q_\alpha
:=
\frac{1}{|I(\alpha)|}
\sum_{r\in I(\alpha)}u_r.
\]
Then
$
q_\alpha\in\partial\rho(\alpha)
$
and
\[
\nabla F_\varepsilon(\alpha)\longrightarrow q_\alpha, \textrm{ as } \varepsilon \rightarrow 0 ^+.
\]
It follows that, locally uniformly in $x$,
\[
\varphi_\alpha^\varepsilon(x)
\longrightarrow
\rho(x)-\rho(\alpha)
-
\langle q_\alpha,x-\alpha\rangle
=
\rho(x)-\langle q_\alpha,x\rangle.
\]
The limiting functional is precisely the gap between the polyhedral norm and the associated supporting hyperplane which is selected at $\alpha$, and hence is a natural non-smooth Bregman functional associated with $\rho$.

This limiting observation does not, however, turn Corollary~\ref{cor:polyhedralbregmanwitness-esoteric}
directly into a pinned-distance theorem for the original polyhedral norm. When $\alpha\neq0$, positive homogeneity gives
\[
\rho(t\alpha)
-
\langle q_\alpha,t\alpha\rangle
=
0
\qquad
\text{for every }t\geq0.
\]
Thus the limiting functional generally vanishes along an entire cone,
rather than only at the pin. Geometrically, the smooth Bregman spheres
flatten in the supporting directions of the polytope as
$\varepsilon\downarrow0$. The case $\alpha=0$ is exceptional: central
symmetry gives $q_0=0$, and hence
\[
\varphi_0^\varepsilon(x)
=
F_\varepsilon(x)
\longrightarrow
\rho(x).
\]
Accordingly, the corollary is most naturally understood as a
positive-radius theorem for the smooth, strictly convex Bregman
geometries generated by $F_\varepsilon$, while the polyhedral limit
identifies the supporting geometry which those functionals
approximate.

We also obtain the following stability result for pinned Bregman functionals.

\begin{corollary}[Stable Neighbourhoods of Bregman Pins]
\label{cor:stablepolyhedralbregmanpins-esoteric}
Assume the hypotheses of Corollary~\ref{cor:polyhedralbregmanwitness-esoteric}.  Then there exist open
neighbourhoods
$$
\alpha_j\in U_j\subseteq\mathbb{R}^d,
\qquad
j=1,\ldots,d,
$$
such that the following holds.  For every Borel $1$-rectifiable set
$E\subseteq\mathbb{R}^d$ satisfying $\mathcal{H}^1(E)>0$, there
exists an index $j_E\in\{1,\ldots,d\}$ for which
$$
\mathcal{H}^1
\big(
\varphi_\beta^\varepsilon(E)
\big)
>0
\qquad
\text{for every }\beta\in U_{j_E}.
$$
The neighbourhoods $U_1,\ldots,U_d$ depend only upon the initial
tuple $\boldsymbol{\alpha}$, the norm $\rho$, and $\varepsilon$, while the successful index
$j_E$ may depend on $E$.
\end{corollary}

\begin{proof}
Affine independence is an open condition.  Since
$
\nabla F_\varepsilon
$
is continuous, we may choose the neighbourhoods $U_1,\ldots,U_d$ so
that
$$
\nabla F_\varepsilon(\beta_1),
\ldots,
\nabla F_\varepsilon(\beta_d)
$$
are affinely independent whenever
$
\beta_j\in U_j
$
for every $j$.  In other words, every tuple
$$
\boldsymbol{\beta}
:=
(\beta_1,\ldots,\beta_d)
\in
U_1\times\cdots\times U_d
$$
is admissible by Proposition~\ref{prop:bregmanfold-esoteric}.

If the conclusion failed for some set $E$, then, for each
$j\in\{1,\ldots,d\}$, we could choose
$
\beta_j\in U_j
$
such that
$$
\mathcal{H}^1
\big(
\varphi_{\beta_j}^\varepsilon(E)
\big)
=0.
$$
The resulting tuple $\boldsymbol{\beta}$ would be admissible,
contradicting Theorem~\ref{thm:dcurvesrefined-genexample}.\qedhere
\end{proof}

The openness in Corollary~\ref{cor:stablepolyhedralbregmanpins-esoteric} also has the following
geometric interpretation in terms of smooth perturbations
$s\mapsto\alpha(s)$.

\begin{proposition}[Smooth Variation of Bregman Spheres]
\label{prop:smoothvariationbregmanspheres-esoteric}
Fix $\varepsilon>0$, and let $F_\varepsilon$ be as defined in \eqref{eq:polyhedrallogsumexp-esoteric}. Suppose that
$
\alpha:I\longrightarrow\mathbb{R}^d
$
is a $C^{\infty}$ deformation of pins, where
$I\subseteq\mathbb{R}$ is an open interval. For each $s\in I$ and
$t>0$, set
$$
\mathcal{S}_{s,\varepsilon}(t)
:= \mathcal{S}_{\alpha(s),\varepsilon} (t) :=
\left\{
x\in\mathbb{R}^d:
\varphi_{\alpha(s)}^\varepsilon(x)=t
\right\}.
$$
Then the following statements hold.

\begin{enumerate}
\item For every $s_0\in I$ and $t>0$, there is an open neighbourhood
$I_0\subseteq I$ of $s_0$ and a $C^\infty$ family of diffeomorphisms
\[
\Theta_s:
\mathcal{S}_{s_0,\varepsilon}(t)
\longrightarrow
\mathcal{S}_{s,\varepsilon}(t),
\qquad s\in I_0,
\]
such that $\Theta_{s_0}$ is the identity. Geometrically, this implies that the Bregman spheres vary smoothly under the perturbation $\alpha$.
\medskip
\item If $x(s)\in\mathcal{S}_{s,\varepsilon}(t)$ is any smooth path,
then its normal velocity is determined by
\begin{equation}
\left\langle
\nu_{s,\varepsilon}(x(s)),
\dot{x}(s)
\right\rangle
=
\frac{
\left\langle
D^2F_\varepsilon(\alpha(s))\dot{\alpha}(s),
x(s)-\alpha(s)
\right\rangle
}{
\left|
\nabla F_\varepsilon(x(s))
-
\nabla F_\varepsilon(\alpha(s))
\right|
},
\label{eq:bregmanspherenormalvelocity-esoteric}
\end{equation}
where
\[
\nu_{s,\varepsilon}(x)
:=
\frac{
\nabla F_\varepsilon(x)
-
\nabla F_\varepsilon(\alpha(s))
}{
\left|
\nabla F_\varepsilon(x)
-
\nabla F_\varepsilon(\alpha(s))
\right|
}
\]
is the outward unit normal.
\medskip
\item For each fixed $s$, the recentred and rescaled Bregman spheres satisfy
\[
\frac{
\mathcal{S}_{s,\varepsilon}(t)-\alpha(s)
}{
\sqrt{2t}
}
\longrightarrow
\underbrace{\left\{
z\in\mathbb{R}^d:
\left\langle
D^2F_\varepsilon(\alpha(s))z,z
\right\rangle=1
\right\}}_{\mathcal{E}_{s,\varepsilon}}, \text{ as } t \rightarrow 0^+.
\]
 The convergence is $C^\infty$ as a convergence of
embedded hypersurfaces and is locally uniform in $s$. The shape of the
limiting ellipsoid depends smoothly on $s$ through $\alpha(s)$.
\end{enumerate}
\end{proposition}

\begin{proof}
For notational convenience, write
$
\alpha_s:=\alpha(s)
$
and define, for each $t>0$, the quantity
$$
G_t(s,x)
:=
\varphi_{\alpha_s}^\varepsilon(x)-t.
$$
Differentiation in the spatial variable gives
\begin{equation}
\nabla_xG_t(s,x)
=
\nabla F_\varepsilon(x)
-
\nabla F_\varepsilon(\alpha_s).
\label{eq:bregmanspherespatialgradient-esoteric}
\end{equation}
Since $D^2F_\varepsilon$ is positive definite everywhere, the gradient mapping $\nabla F_\varepsilon$ is injective. Moreover, $\varphi_{\alpha_s}^\varepsilon$ is non-negative and vanishes only at $\alpha_s$. Therefore, if
$
G_t(s,x)=0
$
and $t>0$, then $x\neq\alpha_s$, and hence
$
\nabla_xG_t(s,x)\neq0.
$
The Implicit Function Theorem consequently shows that
$$
\mathcal{M}_t
:=
\left\{
(s,x)\in I\times\mathbb{R}^d:
G_t(s,x)=0\right\}
=
\left\{(s,x) \in I \times \mathbb{R}^d : \varphi_{\alpha_s}^{\varepsilon} (x) = t
\right\}
$$
is a $C^\infty$ embedded hypersurface in $I\times\mathbb{R}^d \subseteq \mathbb{R}\times \mathbb{R}^d$.

We next compute the variation with respect to the pin $\alpha(s)$. If
$h\in\mathbb{R}^d$, then
\begin{align*}
D_\alpha\varphi_\alpha^\varepsilon(x)[h]
&=
-
\left\langle
\nabla F_\varepsilon(\alpha),h
\right\rangle
-
\left\langle
D^2F_\varepsilon(\alpha)h,x-\alpha
\right\rangle
+
\left\langle
\nabla F_\varepsilon(\alpha),h
\right\rangle\\[1ex]
&=
-
\left\langle
D^2F_\varepsilon(\alpha)h,x-\alpha
\right\rangle.
\end{align*}
It follows that
\begin{equation}\label{eq:bregmansphereparameterderivative-esoteric}
\partial_sG_t(s,x)
=
-
\left\langle
D^2F_\varepsilon(\alpha_s)\dot{\alpha}(s),
x-\alpha_s
\right\rangle.
\end{equation}
Letting
$$
A(s,x)
:=
\left\langle
D^2F_\varepsilon(\alpha_s)\dot{\alpha}(s),
x-\alpha_s
\right\rangle
$$
and defining, along the hypersurface $\mathcal{M}_t$, the vector
$$
V(s,x)
:=
\frac{
A(s,x)
}{
\left|
\nabla F_\varepsilon(x)
-
\nabla F_\varepsilon(\alpha_s)
\right|^2
}
\left(
\nabla F_\varepsilon(x)
-
\nabla F_\varepsilon(\alpha_s)
\right),
$$
the expressions
\eqref{eq:bregmanspherespatialgradient-esoteric}
and
\eqref{eq:bregmansphereparameterderivative-esoteric}
thus give
$$
\partial_sG_t(s,x)
+
\left\langle
\nabla_xG_t(s,x),V(s,x)
\right\rangle
=
-A(s,x)+A(s,x)
=
0.
$$
The vector field
$$
W(s,x):=(1,V(s,x))
$$
is therefore tangent to $\mathcal{M}_t$. Its integral curves advance
the parameter $s$ at unit speed and carry one level hypersurface to
the nearby level hypersurfaces.

To justify this flow uniformly near a fixed $s_0$, observe that
$$
\xi_s
:=
\nabla F_\varepsilon(\alpha_s)
\in\operatorname{int}(P^\circ).
$$
After restricting to a sufficiently small neighbourhood of $s_0$, the vectors $\xi_s$ lie in a compact subset of
$\operatorname{int}(P^\circ)$. Hence there is a constant $c>0$ such that
$$
\rho(x)-\langle\xi_s,x\rangle
\geq c|x|
$$
for every $x\in\mathbb{R}^d$ and every such $s$. Since
$$
F_\varepsilon(x)
\geq
\rho(x)-\varepsilon\log N,
$$
the functions
$\varphi_{\alpha_s}^\varepsilon(x)$ tend to infinity as
$|x|\to\infty$, uniformly for $s$ near $s_0$. The corresponding level hypersurfaces therefore lie in one common compact set. On that compact
set, the denominator in the definition of $V$ is bounded away from
zero along $\mathcal{M}_t$. Thus $W$ is smooth on the relevant compact
portion of $\mathcal{M}_t$. The standard local existence and uniqueness
theorem for smooth ordinary differential equations (see
\cite[Theorem 9.48]{Lee2013}) now gives a local flow on
$\mathcal{M}_t$. Since its $s$-component is identically one, we are guaranteed that this flow
produces the family of diffeomorphisms $\Theta_s$. This proves part
(1).

If $x(s)$ is any smooth path with
$
G_t(s,x(s))=0,
$
then differentiation gives
$$
\partial_sG_t(s,x(s))
+
\left\langle
\nabla_xG_t(s,x(s)),\dot{x}(s)
\right\rangle = 0.
$$
Substituting
\eqref{eq:bregmanspherespatialgradient-esoteric}
and
\eqref{eq:bregmansphereparameterderivative-esoteric},
and then dividing by
$$
\left|
\nabla F_\varepsilon(x(s))
-
\nabla F_\varepsilon(\alpha_s)
\right|,
$$
gives
\eqref{eq:bregmanspherenormalvelocity-esoteric}.
The tangential component of $\dot{x}(s)$ only changes the
parametrization of the sphere. Part (2) of Proposition~\ref{prop:smoothvariationbregmanspheres-esoteric} is thus verified.

It remains to verify part (3). Put $r:=\sqrt{2t}$ and define
$$
Q(s,r,z)
:=
2\int_0^1(1-u)
\left\langle
D^2F_\varepsilon(\alpha_s+urz)z,z
\right\rangle,du.
$$
Linearization thus gives
$$
Q(s,r,z)
=
\frac{
\varphi_{\alpha_s}^\varepsilon(\alpha_s+rz)
}{t}
$$
when $r>0$, while
$$
Q(s,0,z)
=
\left\langle
D^2F_\varepsilon(\alpha_s)z,z
\right\rangle.
$$
In particular, $Q$ is smooth in $(s,r,z)$ down to $r=0$. Since
$D^2F_\varepsilon(\alpha_s)$ is positive definite, the level set
$Q(s,r,z)=1$ is a smooth normal graph over
$$
\left\{
z:
\left\langle
D^2F_\varepsilon(\alpha_s)z,z
\right\rangle=1
\right\}
$$
for all sufficiently small $r$, by the Implicit Function Theorem. The
smooth dependence of $Q$ proves the asserted $C^\infty$ convergence,
locally uniformly in $s$.\qedhere
\end{proof}

There is also a useful description of the exceptional parameters.
It is most transparent in the dual variable
$
\beta=\nabla F_\varepsilon(\alpha).
$
For $\beta\in\mathcal{A}_\varepsilon$, define
$$
\psi_\beta^\varepsilon(x)
:=
F_\varepsilon(x)-\langle \beta,x\rangle.
$$
If $\beta=\nabla F_\varepsilon(\alpha)$, then
\begin{equation}
\varphi_\alpha^\varepsilon(x)
=
\psi_\beta^\varepsilon(x)
-
F_\varepsilon(\alpha)
+
\langle \beta,\alpha\rangle.
\label{eq:bregmandualparameter-esoteric}
\end{equation}
The final two terms are independent of $x$, and therefore do not
affect the $\mathcal{H}^1$ measure of the image of a set.

We now state our estimate for the exceptional set of bad witness parameters for the polyhedral Bregman functionals.

\begin{corollary}[Exceptional Parameters for Polyhedral Bregman
Functionals]
\label{cor:polyhedralbregmanexceptionalset-esoteric}
Let $E\subseteq\mathbb{R}^d$ be a Borel $1$-rectifiable set satisfying
$\mathcal{H}^1(E)>0$.  Define
$$
\mathcal{B}_{E,\varepsilon}^{*}
:=
\left\{
\beta\in\mathcal{A}_\varepsilon:
\mathcal{H}^1\big(\psi_\beta^\varepsilon(E)\big)=0
\right\}.
$$
Then there exists an affine subspace
$
V_{E,\varepsilon}\subseteq\mathbb{R}^d
$
such that
$$
\mathcal{B}_{E,\varepsilon}^{*}
\subseteq
V_{E,\varepsilon}
\qquad
\text{and}
\qquad
\dim V_{E,\varepsilon}\leq d-2.
$$
Equivalently, if
$$
\mathcal{B}_{E,\varepsilon}
:=
\left\{
\alpha\in\mathbb{R}^d:
\mathcal{H}^1\big(\varphi_\alpha^\varepsilon(E)\big)=0
\right\},
$$
then
$$
\mathcal{B}_{E,\varepsilon}
\subseteq \big\{y \in \mathbb{R}^d : \nabla F_{\varepsilon} (y) \in V_{E,\varepsilon}\big\}.
$$
The set on the right-hand side is either empty, or else is a smooth embedded submanifold of dimension at most $d-2$.
\end{corollary}

\begin{proof}
If
$
\boldsymbol{\beta} := (\beta_1,\ldots,\beta_d) \in\mathcal{A}_\varepsilon^d
$
is affinely independent, choose
$
\alpha_j=(\nabla F_\varepsilon)^{-1}(\beta_j).
$
The derivatives
$
D_x\psi_{\beta_j}^\varepsilon
$
and
$
D_x\varphi_{\alpha_j}^\varepsilon
$
then agree pointwise, and Proposition~\ref{prop:bregmanfold-esoteric} therefore shows that every affinely independent $d$-tuple of parameters in $\mathcal{A}_\varepsilon$ is admissible.  Proposition~\ref{prop:parameterizedexceptionalset-genexample} then provides the affine subspace $V_{E,\varepsilon}$.

Equation \eqref{eq:bregmandualparameter-esoteric} and the injectivity of
$\nabla F_\varepsilon$ show that
$
\nabla F_\varepsilon \big( \mathcal{B}_{E,\varepsilon} \big) = \mathcal{B}_{E,\varepsilon}^{*},
$
and the asserted containment thus follows.  Finally,
$
\nabla F_\varepsilon
$
is a diffeomorphism onto the open set
$
\mathcal{A}_\varepsilon.
$
The inverse image of
$
V_{E,\varepsilon}\cap\mathcal{A}_\varepsilon
$
is therefore either empty or a smooth embedded submanifold of the same
dimension as this intersection, which is at most $d-2$.\qedhere
\end{proof}

\begin{remark}
The functionals
$
\varphi_\alpha^\varepsilon
$
are distance-like in that they are non-negative and vanish precisely
at the pin $\alpha$.  In general, however, a Bregman functional is not
symmetric in $x$ and $\alpha$ and does not satisfy the triangle
inequality.  Thus the construction above is not a theorem about
literal distances generated by the smoothed norm.  Its advantage is
that it works uniformly for every polyhedral norm and makes the fold
calculation depend only upon the positive-definiteness of the Hessian.

More generally, Proposition~\ref{prop:bregmanfold-esoteric} applies to any
$C^2$ regularization of a polyhedral norm whose Hessian is positive
definite everywhere.  The latter qualification is necessary for this
proof.  For example, the usual standard convolution with a compactly supported mollifier may leave the regularized function affine on portions of the unbounded polyhedral chambers, so that its Hessian still vanishes
there.  The log-sum-exp regularization was chosen because
\eqref{eq:polyhedrallogsumexphessian-esoteric} verifies the required positive-definiteness globally and explicitly, and therefore introduces no flatness or additional singularity to the argument.

The Hessian of $F_\varepsilon$ need not admit a uniform positive lower bound on all of $\mathbb{R}^d$.  This causes no difficulty for the pointwise fold calculation above.  On every fixed compact set, continuity and positive-definiteness give a positive lower bound for the least eigenvalue, which is the form needed for compactly localized stability arguments.
\end{remark}

\subsection{Fold non-degeneracy loci with arbitrary prescribed \texorpdfstring{$C^2$}{C2} geometry}

The final example departs from the distance-like functionals inspired
by \cite{IL2005,KL2006}. It shows that fold non-degeneracy can hold
when the critical hypersurface is an arbitrary $C^2$ graph over
$\mathbb{R}^{d-1}$.

Let
$
g:\mathbb{R}^{d-1}\rightarrow\mathbb{R}
$
be a fixed $C^2$ function, and decompose $\mathbb{R}^d$ as
$$
x :=(x',t)\in\mathbb{R}^{d-1}\times\mathbb{R}.
$$
We then set
$
u(x',t):=t-g(x'),
$
and, for each parameter $\alpha \in\mathbb{R}^{d-1}$, define
$$
\varphi_{\alpha}^g(x',t)
:= \langle\alpha,x'\rangle
+
u(x',t)^2.
$$
For every affinely independent parameter tuple, the encoding has the
graph of $g$ as its critical set and is fold non-degenerate there.
Thus the singular geometry may be prescribed without losing fold
non-degeneracy.

\begin{proposition}[A Prescribed Fold Hypersurface]
\label{prop:prescribedfoldhypersurface-esoteric}
Let
$
\boldsymbol{\alpha} := (\alpha_1,\ldots,\alpha_d) \in \big(\mathbb{R}^{d-1}\big)^d
$
be an affinely independent parameter tuple, and let
$$
\mathsf{H}_{\boldsymbol{\alpha}}^g
:=
\big(
\varphi_{\alpha_1}^g,\ldots,\varphi_{\alpha_d}^g
\big).
$$
Then
$$
\Sigma_{\boldsymbol{\alpha}}^g
=
\big\{
(x',t)\in\mathbb{R}^d:
t=g(x')
\big\}.
$$
Every point of $\Sigma_{\boldsymbol{\alpha}}^g$ is fold non-degenerate. Consequently, the family
$$
\mathsf{\Phi}_{\boldsymbol{\alpha}}^g
:=
\big\{
\varphi_{\alpha_1}^g,\ldots,\varphi_{\alpha_d}^g
\big\}
$$
is admissible.
\end{proposition}

\begin{proof}
For this $d$-tuple $\boldsymbol{\alpha} \in (\mathbb{R}^{d-1})^{d}$, define the associated $d\times d$ matrix
$$
\mathsf{C}_{\boldsymbol{\alpha}}
:=
\begin{pmatrix}
\alpha_1^{\mathsf{T}}&1 \\
\alpha_2^{\mathsf{T}}&1 \\
\vdots&\vdots \\
\alpha_d^{\mathsf{T}}&1
\end{pmatrix}.
$$
The affine independence of $\alpha_1,\ldots,\alpha_d$ is equivalent to
$
\det\mathsf{C}_{\boldsymbol{\alpha}}\neq0.
$
Moreover, the construction of $\mathsf{C}_{\boldsymbol{\alpha}}$ gives
$$
\mathsf{H}_{\boldsymbol{\alpha}}^g(x',t)
=
\mathsf{C}_{\boldsymbol{\alpha}}
\begin{pmatrix}
x'\\
u(x',t)^2
\end{pmatrix}.
$$
Therefore, a direct calculation of the derivative gives
\begin{equation}\label{eq:prescribedfoldderivative-esoteric}
D_x\mathsf{H}_{\boldsymbol{\alpha}}^g(x',t)
=
\mathsf{C}_{\boldsymbol{\alpha}}
\begin{pmatrix}
I_{d-1}&0\\
-2u(x',t)\nabla g(x')^{\mathsf{T}}&2u(x',t)
\end{pmatrix}.
\end{equation}
Taking determinants in \eqref{eq:prescribedfoldderivative-esoteric} gives
\begin{equation}\label{eq:prescribedfolddeterminant-esoteric}
\det D_x\mathsf{H}_{\boldsymbol{\alpha}}^g(x',t)
=
2\det(\mathsf{C}_{\boldsymbol{\alpha}})
\big(t-g(x')\big).
\end{equation}
It follows immediately that the critical set is precisely the prescribed graph
$$
\Sigma_{\boldsymbol{\alpha}}^g
=
\{(x',t):t=g(x')\}.
$$

At a point of this graph, the second matrix on the right-hand side of
\eqref{eq:prescribedfoldderivative-esoteric} has rank $d-1$ and kernel
$
\operatorname{span}\{e_d\}.
$
Since $\mathsf{C}_{\boldsymbol{\alpha}}$ is invertible, we obtain
$$
\rank D_x\mathsf{H}_{\boldsymbol{\alpha}}^g(x',g(x'))=d-1
$$
and
$$
\ker D_x\mathsf{H}_{\boldsymbol{\alpha}}^g(x',g(x'))
=
\operatorname{span}\{e_d\}.
$$
Finally, \eqref{eq:prescribedfolddeterminant-esoteric} gives
$$
D_x\big(
\det D_x\mathsf{H}_{\boldsymbol{\alpha}}^g
\big)(x',g(x'))[e_d]
=
2\det(\mathsf{C}_{\boldsymbol{\alpha}})
\neq0.
$$
Thus every critical point is fold non-degenerate. By Proposition~\ref{prop:kerneltransversality-diffgeo}, the same calculation shows
directly that the kernel direction is transverse to the prescribed
critical graph.\qedhere
\end{proof}

The results of Section~\ref{sec:generalizedcurveprojections-genexample}
give the following good-witness theorem for the generalized curve
projections $\varphi_{\alpha}^g$ associated with arbitrary
$C^2$-graphs.

\begin{corollary}[Exceptional Parameters for a Prescribed Fold]
\label{cor:prescribedfoldexceptionalset-esoteric}
Let $E\subseteq\mathbb{R}^d$ be a Borel $1$-rectifiable set satisfying
$\mathcal{H}^1(E)>0$, and define
$$
\mathcal{B}_E^g
:=
\left\{
\alpha\in\mathbb{R}^{d-1}:
\mathcal{H}^1
\big(
\varphi_\alpha^g(E)
\big)
=0
\right\}.
$$
Then $\mathcal{B}_E^g$ is contained in an affine subspace of
$\mathbb{R}^{d-1}$ of dimension at most $d-2$.
\end{corollary}

\begin{proof}
Proposition~\ref{prop:prescribedfoldhypersurface-esoteric} shows that every
affinely independent $d$-tuple in the parameter space
$\mathbb{R}^{d-1}$ is admissible. Proposition~\ref{prop:parameterizedexceptionalset-genexample} therefore applies.\qedhere
\end{proof}

The $C^2$ function $g : \mathbb{R}^{d-1} \rightarrow \mathbb{R}$ is arbitrary. Thus the critical set may have any geometric description allowed under $C^2$ graphs. The key idea is that the determinant calculation is entirely unchanged: the missing
direction is always the vertical direction, and the determinant crosses the critical graph linearly in that direction. This example isolates the differential-geometric content of the fold non-degeneracy hypothesis from any additional geometric context: the argument requires only $C^2$ regularity of $g$.

\subsection{A comparison of methods}
\label{subsec:comparisonofmethods-esoteric}

The concrete arguments for pinned distances and radial projections, together with the general results of the present paper, are organized around the same canonical encoding principle: one combines $d$ scalar mappings into the canonical encoding map $\mathsf{H}:=(\varphi_1,\ldots,\varphi_d)$. On the non-critical region, the Inverse Function Theorem makes $\mathsf{H}$ locally bilipschitz, after which Federer's Projection Theorem applies to its coordinate functions; see Proposition~\ref{prop:manifoldinversefunction-diffgeo}. The ingredient concerning the geometry of the critical region is as follows. If the critical set is a $C^1$ hypersurface and $D_x\mathsf{H}$ has rank $d-1$ on its tangent space, the argument can instead be performed intrinsically along that hypersurface, by Proposition~\ref{prop:restrictedrankequivalences-diffgeo}. Fold non-degeneracy is the stronger condition which makes this geometry stable under perturbations of the parameters; its differential-geometric content is summarized by Proposition~\ref{prop:kerneltransversality-diffgeo}.

The pinned-distance calculation in Section~\ref{sec:pinneddistances-examples} is the clean and canonical model in which the critical set is a hyperplane. The framework developed in the present paper treats the portion lying in the critical hyperplane as well, since the tangential-immersion condition supplies the required lower-dimensional change of variables. It also separates the differential-geometric calculation for one tuple from the affine argument controlling the exceptional parameter set.

The radial-projection calculation in Section~\ref{sec:radialprojections-examples} illustrates a different possibility. Along the line joining the two centres, the missing derivative direction is tangent to the critical set, so the tangential-immersion condition fails. The hypothesis requiring positive length away from that line is exactly what allows one to remain in the non-critical branch of the method.

Thus invertibility controls the regular region, tangential rank controls the critical region, and the fold non-degeneracy condition controls stability as one varies parameters. This spatial method is complementary to Peres--Schlag transversality \cite{PS2000}, which differentiates a normalized two-point difference in the parameter variable and yields almost-everywhere and exceptional-dimension estimates. The present method differentiates the determinant of a finite encoding in a spatial kernel direction and gives the deterministic conclusion that an admissible finite tuple must contain a good witness.

\vfill
\pagebreak

\appendix

\section{Differential geometry of curve projections}
\label{sec:differentialgeometry-diffgeo}

This appendix collects the differential-geometric facts used in Sections~\ref{sec:generalizedcurveprojections-genexample} and \ref{sec:exoticcurveprojections-esoteric}. In particular, it makes precise the critical-hypersurface and tangential-immersion conditions appearing in Definition~\ref{def:criticalhypersurfacetangentialimmersion-genexample}, explains the relation between tangential immersion, nonzero minors, and kernel transversality, and records the manifold form of the Inverse Function Theorem used in the proof of Theorem~\ref{thm:dcurvesrefined-genexample}. The conventions for embedded submanifolds and tangent spaces follow \cite{Lee2013}.

\subsection{Embedded submanifolds and restricted rank}
\label{subsec:manifoldnotation-diffgeo}

Let $M\subseteq\mathbb{R}^d$ be a $C^1$ embedded submanifold of dimension $k\in\{1,\ldots,d-1\}$, and let $p\in M$. We identify the tangent space $T_pM$ with the $k$-dimensional linear subspace
$$
T_pM
:=
\big\{
\gamma'(0):
\gamma:(-\varepsilon,\varepsilon)\rightarrow M
\text{ is }C^1,\ \gamma(0)=p
\big\}
\subseteq\mathbb{R}^d.
$$
Its dual space is denoted by $(T_pM)^*$, and always satisfies
$$
\dim T_pM=\dim(T_pM)^*=k.
$$

If $\mathsf{H}:\mathbb{R}^d\rightarrow\mathbb{R}^{\ell}$ is $C^1$, then its derivative restricted to $M$ at $p$ is
$$
D_x\mathsf{H}(p)\vert_{T_pM}
:
T_pM\longrightarrow\mathbb{R}^{\ell}.
$$
The rank of this restricted map always satisfies
$$
\rank\big(D_x\mathsf{H}(p)\vert_{T_pM}\big)
=
\dim D_x\mathsf{H}(p)(T_pM)
\leq k.
$$
If $\varphi:\mathbb{R}^d\rightarrow\mathbb{R}$ is $C^1$, then its restricted differential is the covector
$$
D_x\varphi(p)\vert_{T_pM}
\in(T_pM)^*,
\text{ where }
D_x\varphi(p)[v]
=
\langle\nabla\varphi(p),v\rangle,
\qquad v\in T_pM.
$$

We first record the manifold form of the Inverse Function Theorem used in the proof of Theorem~\ref{thm:dcurvesrefined-genexample}; see \cite{Lee2013} for a standard reference.

\begin{proposition}[Manifold Inverse Function Theorem]
\label{prop:manifoldinversefunction-diffgeo}
Let $M$ and $N$ be $C^1$ manifolds of the same dimension $k$, and let $F:M\rightarrow N$ be a $C^1$ map. If
$$
DF(p):T_pM\rightarrow T_{F(p)}N
$$
is an isomorphism, there exist neighbourhoods $p\in U\subset M$ and $F(p)\in V\subset N$ such that
$
F\vert_U:U\rightarrow V
$
is a $C^1$ diffeomorphism.

Moreover, if $M$ and $N$ are embedded in Euclidean spaces, then $U$ and $V$ may be chosen simultaneously of small diameter so that this local diffeomorphism and its inverse are Lipschitz with respect to the ambient Euclidean metrics.
\end{proposition}

\begin{proof}
The first conclusion is the ordinary Inverse Function Theorem applied to a coordinate representation of $F$. After restricting the coordinate domains to sufficiently small relatively compact neighbourhoods, the derivatives of the coordinate representation and its inverse are bounded. The corresponding coordinate charts and their inverses also have bounded derivatives on these neighbourhoods. After choosing convex coordinate neighbourhoods, the Mean Value Theorem therefore gives finite Lipschitz constants for $F\vert_U$ and $(F\vert_U)^{-1}$.\qedhere
\end{proof}

The following linear-algebraic reformulation isolates the exact calculation which is used repeatedly in Section~\ref{sec:generalizedcurveprojections-genexample}.

\begin{proposition}[Equivalent Forms of Tangential Immersion]
\label{prop:restrictedrankequivalences-diffgeo}
Let
$
\mathsf{H} = (\varphi_1,\ldots,\varphi_d) : \mathbb{R}^d\rightarrow\mathbb{R}^{d}
$
be any $C^1$ mapping. Suppose that $M \subset \mathbb{R}^d$ is an embedded $k$-dimensional $C^1$ submanifold, and let $p \in M$.

The following conditions are equivalent.
\begin{enumerate}
    \item The restricted derivative of $\mathsf{H}$ has maximal rank when restricted to the tangent space of $M$ at $p$:
    $$
    \rank\big(D_x\mathsf{H}(p)\vert_{T_pM}\big)=k.
    $$

    \item The restricted covectors
    $$
    D_x\varphi_1(p)\vert_{T_pM},\ldots,
    D_x\varphi_d(p)\vert_{T_pM}
    $$
    span $(T_pM)^*$.
    \medskip
    \item For some index set
    $
    I=\{i_1,\ldots,i_k\} \subseteq\{1,\ldots,d\},
    $
    the associated restricted derivative
    $$
    D\big(\mathsf{H}^I\vert_M\big)(p)
    :
    T_pM\rightarrow\mathbb{R}^k,
    \qquad
    \mathsf{H}^I
    :=
    (\varphi_{i_1},\ldots,\varphi_{i_k}),
    $$
    is a linear isomorphism onto $\mathbb{R}^k$.
    \medskip
    \item If $\theta:W\subseteq\mathbb{R}^k\rightarrow M$ is any local parametrization satisfying $\theta(0)=p$, then the associated $d\times k$ derivative matrix
    $$
    D(\mathsf{H}\circ\theta)(0)
    =
    D_x\mathsf{H}(p)D\theta(0)
    $$
    has an invertible $k\times k$ minor.
\end{enumerate}
Whenever these conditions hold, the successful mapping $\mathsf{H}^I\vert_M$ is a local $C^1$ diffeomorphism and is locally bilipschitz after its domain is restricted.
\end{proposition}

\begin{proof}
The rows of $D(\mathsf{H}\circ\theta)(0)$ are the coordinate representations of the covectors
$$
D_x\varphi_j(p)\vert_{T_pM},
\qquad j=1,\ldots,d,
$$
with respect to the basis of $T_pM$ supplied by $D\theta(0)$. Thus conditions (1), (2), and (4) all assert that this matrix has rank $k$. A matrix with $k$ columns has rank $k$ if and only if one of its $k\times k$ minors is nonzero. Choosing the rows in that minor gives the index set $I$ in condition (3), and the converse is immediate. The final conclusion is Proposition~\ref{prop:manifoldinversefunction-diffgeo}.\qedhere
\end{proof}

\subsection{Critical hypersurfaces and kernel transversality}
\label{subsec:hypersurfacetransversality-diffgeo}
We now record the relationship between the critical-hypersurface condition and kernel transversality for the associated derivative.

Suppose that $J:U\rightarrow\mathbb{R}$ is $C^1$ on an open set $U\subseteq\mathbb{R}^d$, that $J(p)=0$, and that $D_xJ(p)\neq0$. The Regular Level Set Theorem gives a neighbourhood $p\in W\subseteq U$ such that
$$
M:=J^{-1}(0)\cap W
$$
is a $C^1$ hypersurface. Its tangent space is given exactly by
\begin{equation}\label{eq:regularleveltangent-diffgeo}
T_pM
=
\ker D_xJ(p)
=
\big\{
v\in\mathbb{R}^d:
\langle\nabla J(p),v\rangle=0
\big\}.
\end{equation}
This is the form of the Regular Level Set Theorem used for the critical hypersurfaces in Section~\ref{subsec:foldstability-genexample}; see \cite{Lee2013}.

\begin{proposition}[Kernel Transversality Criterion]
\label{prop:kerneltransversality-diffgeo}
Let $U\subseteq\mathbb{R}^d$ be open, let $\mathsf{H}:U\rightarrow\mathbb{R}^d$ be $C^2$, and define
$
J(x):=\det D_x\mathsf{H}(x).
$
Suppose that $x_0\in U$ is a fixed point which satisfies
$$
\rank D_x\mathsf{H}(x_0)=d-1,
\qquad
D_xJ(x_0)\neq0.
$$
After shrinking to a neighbourhood $W$ of $x_0$, let
$$
M:=J^{-1}(0)\cap W,
$$
which is a $C^1$ hypersurface by the Regular Level Set Theorem.
The following conditions are equivalent.
\begin{enumerate}
\item The restricted derivative map at $x_0$ has maximal rank:
$$
\rank\big(D_x\mathsf{H}(x_0)\vert_{T_{x_0}M}\big)=d-1.
$$
\item The kernel of the derivative matrix at $x_0$ intersects the tangent space of $M$ at $x_0$ transversely:
$$
\ker D_x\mathsf{H}(x_0)\cap T_{x_0}M=\{0\}.
$$
\item The fold non-degeneracy condition for spatial determinants holds at $x_0$:
\[
D_xJ(x_0)[v]\neq0,
\qquad \forall v \in \ker D_x\mathsf{H}(x_0) \setminus \{0\}.
\]
\end{enumerate}
\end{proposition}

\begin{proof}
Since $\rank D_x\mathsf{H}(x_0)=d-1$, the Rank-Nullity theorem implies that $\ker D_x\mathsf{H}(x_0)$ is a one-dimensional vector space. The restriction to the $(d-1)$-dimensional space $T_{x_0}M$ therefore has rank $d-1$ if and only if it is injective, which is equivalent to requiring that
$$
\ker D_x\mathsf{H}(x_0)\cap T_{x_0}M=\{0\}.
$$
Identity \eqref{eq:regularleveltangent-diffgeo} then converts the latter condition into
$$
D_xJ(x_0)[v]\neq0
$$
for every $v\in\ker D_x\mathsf{H}(x_0)\setminus \{0\}$.\qedhere
\end{proof}

In particular, if $J=\det D_x\mathsf{H}$, the last displayed condition is exactly the fold kernel-transversality condition \eqref{eq:foldkerneltransversality-genexample}. Proposition~\ref{prop:kerneltransversality-diffgeo} therefore explains directly why fold non-degeneracy implies the tangential-immersion condition.

\bibliographystyle{abbrv}
\bibliography{refs}
\end{document}